\documentclass[12pt]{amsart}
\usepackage{amsmath,amsfonts,amssymb,amsthm,amsbsy,latexsym,cite}
\usepackage[colorlinks,linkcolor=blue,citecolor=blue,menucolor=blue]{hyperref}
\usepackage[latin1]{inputenc}

\multlinegap=\parindent

\begin{document}

\title{Certain residual properties of~generalized~Baumslag--Solitar groups}

\author{E.~V.~Sokolov}
\address{Ivanovo State University, Russia}
\email{ev-sokolov@yandex.ru}

\begin{abstract}
Let~$G$ be~a~generalized Baumslag--Solitar group and~$\mathcal{C}$ be~a~class of~groups containing at~least one non-unit group and~closed under~taking subgroups, extensions, and~Cartesian products of~the~form $\prod_{y \in Y}X_{y}$, where $X$,~$Y \in \mathcal{C}$ and~$X_{y}$ is~an~isomorphic copy of~$X$ for~every $y \in Y$. We give a~criterion for~$G$ to~be residually a~$\mathcal{C}$\nobreakdash-group provided~$\mathcal{C}$ consists only of~periodic groups. We also prove that~$G$ is~residually a~torsion-free $\mathcal{C}$\nobreakdash-group if~$\mathcal{C}$ contains at~least one non-pe\-ri\-od\-ic group and~is closed under~taking homomorphic images. These statements generalize and~strengthen some known results. Using the~first of~them, we provide criteria for~a~GBS\nobreakdash-group to~be a)~residually nilpotent; b)~residually torsion-free nilpotent; c)~residually free.
\end{abstract}


\maketitle

\newtheorem{theorem}{Theorem}
\newtheorem{corollary}{Corollary}
\newtheorem{proposition}{Proposition}[section]
\theoremstyle{definition}
\newtheorem*{algorithm1}{Algorithm}
\newtheorem*{algorithm2}{An~algorithm for~verifying the~condition of~Statement~2\nobreakdash-\textit{c} of~Theorem~\ref{theorem05}}

\vspace{-2pt}

\section*{Introduction}

A~group is~called a~\emph{generalized Baumslag--Solitar group}, or~a~\emph{GBS-group}, if~it~is~the~fundamental group of~a~graph of~groups with~infinite cyclic vertex and~edge groups. GBS-groups have been actively studied in~recent years~\cite{Levitt2007, Levitt2015, DelgadoRobinsonTimm2011, Alonso2012, DelgadoRobinsonTimm2014, DelgadoRobinsonTimm2017, Dudkin2017, Dudkin2018, Dudkin2020}, and~many of~these investigations are devoted to~establishing a~connection between the~algebraic properties of~GBS-groups and~the~structure of~the~graphs defining them. The~aim of~this paper is~to~describe certain residual properties of~a~GBS-group in~terms of~the~associated graph of~groups. We strengthen some known results (for~example, on~the~residual finiteness and~the~residual $p$\nobreakdash-fi\-nite\-ness of~GBS-groups) and~give a~criterion for~the~residual nilpotence of~a~GBS\nobreakdash-group.

Let $\mathcal{C}$ be~a~class of~groups. A~group~$G$ is~said to~be \emph{residually a~$\mathcal{C}$\nobreakdash-group} if, for~any non-unit element $g \in G$, there exists a~homomorphism~$\sigma$ of~$G$ onto~a~group of~$\mathcal{C}$ such that $g\sigma \ne 1$. The~most commonly considered situation is~when $\mathcal{C}$ is~the~class of~all finite groups, all finite $p$\nobreakdash-groups (where $p$ is~a~prime number), all nilpotent groups or~all solvable groups. In~these cases $G$ is~called \emph{residually finite}, \emph{residually $p$\nobreakdash-finite}, \emph{residually nilpotent} or~\emph{residually solvable} respectively.

We say that a~class $\mathcal{C}$ of~groups is~\emph{root} if~it~contains at~least one non-unit group and~is closed under~taking subgroups, extensions, and~Cartesian products of~the~form $\prod_{y \in Y}X_{y}$, where $X$,~$Y \in \mathcal{C}$ and~$X_{y}$ is~an~isomorphic copy of~$X$ for~every $y \in Y$. The~notion of~a~root class was introduced by~Gruenberg~\cite{Gruenberg1957}, and~the~above definition is~equivalent to~that given in~\cite{Gruenberg1957}; see~\cite{Sokolov2015} for~details.

The~classes of~all finite groups, all finite $p$\nobreakdash-groups, all periodic groups of~finite exponent, all solvable groups, and~all tor\-sion-free groups can serve as~examples of~root classes. It~is~also easy to~see that the~intersection of~any number of~root classes is~again a~root class. At~the~same time, the~classes of~all nilpotent groups, all tor\-sion-free nilpotent groups, and~all finite nilpotent groups are not~root because they are not~closed under~taking extensions.

The~main goal of~this paper is~to~get necessary and~sufficient conditions for~a~GBS-group to~be residually a~$\mathcal{C}$\nobreakdash-group, where $\mathcal{C}$ is~an~arbitrary, not~any specific root class. The~sense of~studying residually $\mathcal{C}$\nobreakdash-groups, where $\mathcal{C}$ is~an~arbitrary class of~groups, is~to~get many results at~once using the~same reasoning. This approach was originally proposed in~\cite{Gruenberg1957, Malcev1958} and~turned out to~be very fruitful in~the~study of~free constructions of~groups in~the~case when $\mathcal{C}$ was~a~root class; see,~e.~g.~\cite{AzarovTieudjo2002, Sokolov2015, Tumanova2017, Tumanova2019, SokolovTumanova2019, SokolovTumanova2020IVM, SokolovTumanova2020LJM, SokolovTumanova2020SMJ}.

For~a~root class~$\mathcal{C}$ and~a~GBS\nobreakdash-group~$G$, we give a~criterion for~$G$ to~be residually a~$\mathcal{C}$\nobreakdash-group if~$\mathcal{C}$ consists of~only periodic groups (Theorem~\ref{theorem03}) and~a~sufficient condition for~$G$ to~be residually a~$\mathcal{C}$\nobreakdash-group if~$\mathcal{C}$ contains at~least one non-pe\-ri\-od\-ic group (Theorem~\ref{theorem04}). Using the~first of~these results, we prove criteria for~a~GBS\nobreakdash-group to~be a)~residually nilpotent (Theorem~\ref{theorem05}); b)~residually tor\-sion-free nilpotent and~residually free (Theorem~\ref{theorem06}). All the~proofs use only the~classical methods of~combinatorial group theory and~the~basic concepts of~graph theory.

\section{Statement of~results}\label{section01}

First, we formulate a~number of~known statements on~the~residual properties of~(ordinary) Baumslag--Solitar groups since they complement the~results obtained in~this paper and~are used in~the~proofs of~some of~them.

Recall that a~\emph{Baumslag--Solitar group} is~a~group with~the~presentation
$$
\mathrm{BS}(m,n) = \langle a, b;\ a^{-1}b^{m}a = b^{n}\rangle,
$$
where $m$ and~$n$ are non-zero integers. Since $\mathrm{BS}(m,n)$, $\mathrm{BS}(n,m)$, and~$\mathrm{BS}(-m,-n)$ are pairwise isomorphic, we can assume without loss of~generality that $|n| \geqslant m > 0$.

We also recall that if~$\rho$ is~a~set of~primes, then a~\emph{$\rho$\nobreakdash-num\-ber} is~an~integer, all prime divisors of~which belong to~$\rho$, and~a~\emph{$\rho$\nobreakdash-group} is~a~periodic group, the~orders of~all elements of~which are $\rho$\nobreakdash-num\-bers. If~$\rho$ consists of~one number~$p$, then we write ``$p$\nobreakdash-'' instead of~``$\{p\}$\nobreakdash-''.

For~a~class of~groups~$\mathcal{C}$ consisting only of~periodic groups, let $\rho(\mathcal{C})$ denote the~set of~primes defined as~follows: $p \in \rho(\mathcal{C})$ if~and~only if~there exists a~$\mathcal{C}$\nobreakdash-group~$X$ such that $p$ divides the~order of~some element of~$X$.

\begin{theorem}\label{theorem01}
\textup{\cite{Tumanova2017}}
Let $\mathcal{C}$ be~a~root class of~groups consisting only of~periodic groups and~closed under~taking quotient groups, $\rho(\mathcal{C})$ be~the~set of~primes defined above. Then the~following statements hold.

\textup{1.}\hspace{1ex}If $1 < m < |n|$, then $\mathrm{BS}(m,n)$ is~not~residually a~$\mathcal{C}$\nobreakdash-group.

\textup{2.}\hspace{1ex}$\mathrm{BS}(m,m)$ is~residually a~$\mathcal{C}$\nobreakdash-group if~and~only if~$m$ is~a~$\rho(\mathcal{C})$\nobreakdash-num\-ber.

\textup{3.}\hspace{1ex}$\mathrm{BS}(m,-m)$ is~residually a~$\mathcal{C}$\nobreakdash-group if~and~only if~$m$ is~a~$\rho(\mathcal{C})$\nobreakdash-num\-ber and~$2 \in \rho(\mathcal{C})$.

\textup{4.}\hspace{1ex}$\mathrm{BS}(1,n)$, $|n| \ne 1$, is~residually a~$\mathcal{C}$\nobreakdash-group if~and~only if~there exists $p \in \rho(\mathcal{C})$ not~dividing $n$ and~such that the~order of~the~image $n+p\mathbb{Z}$ of~$n$ in~the~multiplicative group of~$\mathbb{Z}_{p}$ is~a~$\rho(\mathcal{C})$\nobreakdash-num\-ber.
\end{theorem}

We note that, in~fact, Theorem~\ref{theorem01} is~true for~an~arbitrary root class consisting of~periodic groups (see Proposition~\ref{proposition203} below).

\begin{theorem}\label{theorem02}
\textup{\cite{Moldavanskii2020}}
$\mathrm{BS}(m,n)$ is~residually nilpotent if~and~only if~either $m = 1$ and~$n \ne 2$, or~$m > 1$ and~$n = \varepsilon m$ for~some $\varepsilon = \pm 1$.
\end{theorem}

Now we turn to~generalized Baumslag--Solitar groups. For~a~detailed description of~the fundamental groups of~graphs of~groups and~GBS-groups, we refer the~reader to~Sections~\ref{section03} and~\ref{section04}. Here we just recall that each GBS-group can be~defined by~a~\emph{graph with~labels} (which are non-zero integers associated with~the~edges of~the~graph). This graph is~called \emph{reduced} if~each edge that is~not~a~loop has labels different from~$\pm 1$. It~is~easy to~show that any GBS-group can be~defined by~a~reduced labeled graph (see Section~\ref{section04} for~details).

A~GBS group is~called \emph{elementary} if~it~is~isomorphic to~$\mathbb{Z}$, $\mathrm{BS}(1,1) \cong \mathbb{Z} \times \mathbb{Z}$ or~$\mathrm{BS}(1,-1)$ \cite[p.~6]{Levitt2015}. It~is~known that a~GBS\nobreakdash-group is~solvable if~it~is~elementary or~isomorphic to~$\mathrm{BS}(1,q)$, where $q \ne \pm 1$~\cite{DelgadoRobinsonTimm2014}.

Until the~end of~this section, let $\Gamma$ be~a~non-empty finite connected graph, $\mathcal{L}(\Gamma)$ be~a~labeled graph over~$\Gamma$, $G$ be~the~GBS-group defined by~$\mathcal{L}(\Gamma)$, and~$\Delta\colon G \to \mathbb{Q}^{*}$ be~the~\emph{modular homomorphism} of~$G$ (defined if~$G$ is~not~elementary; see Section~\ref{section04}). The~theorems and~corollaries formulated below are the~main results of~this paper.

\begin{theorem}\label{theorem03}
Let $\mathcal{C}$ be~a~root class of~groups consisting only of~periodic groups and~$\rho(\mathcal{C})$ be~the~set of~primes defined above. Suppose also that $G$ is~not~solvable and~$\mathcal{L}(\Gamma)$ is~reduced.

\textup{1.}\hspace{1ex}If $\operatorname{Im}\Delta = \{1\}$, then $G$ is~residually a~$\mathcal{C}$\nobreakdash-group if~and~only if~all the~labels of~$\mathcal{L}(\Gamma)$ are $\rho(\mathcal{C})$\nobreakdash-num\-bers.

\textup{2.}\hspace{1ex}If $\operatorname{Im}\Delta = \{1,-1\}$, then $G$ is~residually a~$\mathcal{C}$\nobreakdash-group if~and~only if~all the~labels of~$\mathcal{L}(\Gamma)$ are $\rho(\mathcal{C})$\nobreakdash-num\-bers and~$2 \in \rho(\mathcal{C})$.

\textup{3.}\hspace{1ex}If $\operatorname{Im}\Delta \not\subseteq \{1,-1\}$, then $G$ is~not~residually a~$\mathcal{C}$\nobreakdash-group.
\end{theorem}

\begin{corollary}\label{corollary01}
The~following statements are equivalent.

\textup{1.}\hspace{1ex}$G$ is~residually finite.

\textup{2.}\hspace{1ex}$G$ is~residually finite solvable.

\textup{3.}\hspace{1ex}Either $G$ is~solvable, or~it~is~not~and~$\operatorname{Im}\Delta \subseteq \{1,-1\}$.
\end{corollary}

\begin{corollary}\label{corollary02}
Let $G$ be~not~solvable, $\mathcal{L}(\Gamma)$ be~reduced, and~$\rho$ be~a~non-empty set of~primes. The~following statements are equivalent.

\textup{1.}\hspace{1ex}$G$ is~residually a~finite $\rho$\nobreakdash-group.

\textup{2.}\hspace{1ex}$G$ is~residually a~finite solvable $\rho$\nobreakdash-group.

\textup{3.}\hspace{1ex}$G$ is~residually a~periodic $\rho$\nobreakdash-group of~finite exponent.

\textup{4.}\hspace{1ex}$G$ is~residually a~periodic solvable $\rho$\nobreakdash-group of~finite exponent.

\textup{5.}\hspace{1ex}$\operatorname{Im}\Delta \subseteq \{1,-1\}$, all the~labels of~$\mathcal{L}(\Gamma)$ are $\rho$\nobreakdash-num\-bers, and~if $-1 \in \operatorname{Im}\Delta$, then $2 \in \rho$.
\end{corollary}

\begin{theorem}\label{theorem04}
{\parfillskip=0pt
Let $\mathcal{C}$ be~a~root class of~groups containing at~least one non-pe\-ri\-od\-ic group.\par}

\textup{1.}\hspace{1ex}If $G$ is~elementary, then it~is~a~tor\-sion-free $\mathcal{C}$\nobreakdash-group.

\textup{2.}\hspace{1ex}Let $G$ be~not~elementary and~$Q$ be~the~subring of~$\mathbb{Q}$ generated by~$\operatorname{Im}\Delta$. If~the~additive group of~$Q$ belongs to~$\mathcal{C}$, then $G$ is~residually a~tor\-sion-free $\mathcal{C}$\nobreakdash-group. In~particular, if~$\operatorname{Im}\Delta \subseteq \{1,-1\}$ or~$\mathcal{C}$ is~closed under~taking quotient groups, then $G$ is~residually a~tor\-sion-free $\mathcal{C}$\nobreakdash-group.
\end{theorem}

\begin{corollary}\label{corollary03}
An~arbitrary GBS\nobreakdash-group is~residually a~tor\-sion-free solvable group.
\end{corollary}

The~largest cyclic normal subgroup of~$G$ is~called the~\emph{cyclic radical} of~this group and~is denoted by~$C(G)$. The~cyclic radical exists if~$G$ is~not~isomorphic to~$\mathrm{BS}(1,1)$ or~$\mathrm{BS}(1,-1)$~\cite[p.~1808]{DelgadoRobinsonTimm2017}.

\begin{theorem}\label{theorem05}
Let $G$ be~not~solvable and~$\mathcal{L}(\Gamma)$ be~reduced.

\textup{1.}\hspace{1ex}If $\operatorname{Im}\Delta = \{1\}$, then $G$ is~residually nilpotent if~and~only if~it~is~residually a~finite $p$\nobreakdash-group for~some prime number~$p$.

\textup{2.}\hspace{1ex}If $\operatorname{Im}\Delta = \{1,-1\}$, then the~following statements are equivalent:

a\textup{)}\hspace{1ex}$G$ is~residually nilpotent;

b\textup{)}\hspace{1ex}$G$ is~residually a~finite nilpotent $\{2,p\}$\nobreakdash-group for~some prime number~$p$ \textup{(}which can be~equal to~$2$\textup{)};

c\textup{)}\hspace{1ex}all the~labels of~$\mathcal{L}(\Gamma)$ are $p$\nobreakdash-num\-bers for~some prime number~$p$, and~if $p \ne 2$, then every elliptic element that is~conjugate to~its inverse belongs to~$C(G)$.

\textup{3.}\hspace{1ex}If $\operatorname{Im}\Delta \not\subseteq \{1,-1\}$, then $G$ is~not~residually nilpotent.
\end{theorem}

The~definition of~elliptic element can be~found in~Section~\ref{section04}. We note that if~$\mathcal{L}(\Gamma)$ is~reduced, all its labels are $p$\nobreakdash-num\-bers for~some prime number $p \ne 2$, and~$\operatorname{Im}\Delta = \{1,-1\}$, then there is~an~algorithm that checks whether every elliptic element that is~conjugate to~its inverse belongs to~$C(G)$; this algorithm is~given at~the~end of~Section~\ref{section06}.

\begin{theorem}\label{theorem06}
Let $G$ be~not~cyclic. The~following statements are equivalent.

\textup{1.}\hspace{1ex}$G$ is~residually tor\-sion-free nilpotent.

\textup{2.}\hspace{1ex}$G$ is~residually free.

\textup{3.}\hspace{1ex}$G$ is~isomorphic to~the~direct product of~a~free group and~an~infinite cyclic group.
\end{theorem}

Thus, Theorems~\ref{theorem01} and~\ref{theorem03} (in~combination with~Proposition~\ref{proposition203}) give a~criterion for~$G$ to~be residually a~$\mathcal{C}$\nobreakdash-group, where $\mathcal{C}$ is~a~root class of~groups consisting only of~periodic groups, while Theorems~\ref{theorem02} and~\ref{theorem05} do a~criterion for~the~residual nilpotence of~$G$ thereby answering~\cite[Question~4]{BardakovNeshchadim2020}. We note that Corollaries~\ref{corollary01},~\ref{corollary02}, and~\ref{corollary03} strengthen and~generalize Corollary~7.7 of~\cite{Levitt2015}, Theorem~1 of~\cite{Dudkin2020}, and~Corollary~3 of~\cite{Robinson2010} respectively. The~rest of~the~paper is~devoted to~the~proofs of~the~formulated statements.

\section{Some auxiliary statements}\label{section02}

Throughout this section, if~$\mathcal{C}$ is~a~class of~groups consisting only of~periodic groups, then $\rho(\mathcal{C})$ denotes the~set of~primes defined above.

\begin{proposition}\label{proposition201}
Let $\mathcal{C}$ be~a~class of~groups consisting only of~periodic groups and~closed under~taking subgroups and~extensions. Then any finite solvable $\rho(\mathcal{C})$\nobreakdash-group belongs to~$\mathcal{C}$.
\end{proposition}

\begin{proof}
Let $X$ be~a~finite solvable $\rho(\mathcal{C})$\nobreakdash-group. Then there~is a~polycyclic series~$\mathcal{S}$ in~$X$ such that the~orders of~all its factors belong to~$\rho(\mathcal{C})$. Let $p$ be~the~order of~some factor~$F$. By~the~definition of~$\rho(\mathcal{C})$, $p$~divides the~order of~an~element of~some $\mathcal{C}$\nobreakdash-group, and~so this group contains an~element, say~$x$, of~order~$p$. Then $F$ is~isomorphic to~the~cyclic subgroup~$\langle x \rangle$ generated by~$x$, and~since $\mathcal{C}$ is~closed under~taking subgroups,  $\langle x \rangle \in \mathcal{C}$. Thus, all the~factors of~$\mathcal{S}$ are $\mathcal{C}$\nobreakdash-groups, and~$X \in \mathcal{C}$ because $\mathcal{C}$ is~closed under~taking extensions.
\end{proof}

\begin{proposition}\label{proposition202}
\textup{\cite[Proposition~17]{SokolovTumanova2020IVM}}
Let $\mathcal{C}$ be~a~root class of~groups consisting only of~periodic groups. Then any $\mathcal{C}$\nobreakdash-group is~of~finite exponent.
\end{proposition}

\begin{proposition}\label{proposition203}
{\parfillskip=0pt
Theorem~\textup{\ref{theorem01}} is~valid for~any root class~$\mathcal{C}$ consisting of~periodic groups.\par}
\end{proposition}

\begin{proof}
Let $\mathcal{C}_{1}$ and~$\mathcal{C}_{2}$ denote the~class of~finite solvable $\rho(\mathcal{C})$\nobreakdash-groups and~the~class of~periodic $\rho(\mathcal{C})$\nobreakdash-groups of~finite exponent respectively. It~follows from~Propositions~\ref{proposition201} and~\ref{proposition202} that $\mathcal{C}_{1} \subseteq \mathcal{C} \subseteq \mathcal{C}_{2}$. One can easily verify that $\mathcal{C}_{1}$ and~$\mathcal{C}_{2}$ are root classes closed under~taking quotient groups. It~is~also obvious that $\rho(\mathcal{C}_{1}) = \rho(\mathcal{C}) = \rho(\mathcal{C}_{2})$. Therefore, if~$G$ is~residually a~$\mathcal{C}$\nobreakdash-group, then it~is~residually a~$\mathcal{C}_{2}$\nobreakdash-group and~satisfies the~necessary conditions from~Theorem~\ref{theorem01} (depending only on~$m$, $n$, and~$\rho(\mathcal{C}))$. In~the~same way, if~the~sufficient conditions from~Theorem~\ref{theorem01} are satisfied (which also depend only on~$m$, $n$, and~$\rho(\mathcal{C}))$, then $G$ is~residually a~$\mathcal{C}_{1}$\nobreakdash-group and~hence is~residually a~$\mathcal{C}$\nobreakdash-group.
\end{proof}

\begin{proposition}\label{proposition204}
Let $\mathcal{C}$ be~an~arbitrary root class of~groups. Then the~following statements are true.

\textup{1.}\hspace{1ex}Every free group is~residually a~$\mathcal{C}$\nobreakdash-group.

\textup{2.}\hspace{1ex}The~direct product of~any two residually $\mathcal{C}$\nobreakdash-groups is~residually a~$\mathcal{C}$\nobreakdash-group.

\textup{3.}\hspace{1ex}Any extension of~a~residually $\mathcal{C}$\nobreakdash-group by~a~$\mathcal{C}$\nobreakdash-group is~residually a~$\mathcal{C}$\nobreakdash-group.
\end{proposition}

\begin{proof}
Statements~1 and~3 follow from~\cite[Theorem~1]{AzarovTieudjo2002} and~\cite[Lemma~1.5]{Gruenberg1957} respectively. Statement~2 is~verified directly.
\end{proof}

As~usual, by~a~\emph{group of~prime power order} we mean a~finite group whose order is~a~power of~some prime number.

\begin{proposition}\label{proposition205}
Let $X$ be~a~finitely generated group. If~$X$ is~residually nilpotent, then it~is~residually a~group of~prime power order. If~$X$ is~residually tor\-sion-free nilpotent, then it~is~residually a~finite $p$\nobreakdash-group for~every prime number~$p$.
\end{proposition}

\begin{proof}
Since $X$ is~finitely generated, it~is~residually a~finitely generated nilpotent group or~residually a~finitely generated tor\-sion-free nilpotent group. Therefore, the~required statement follows from~\cite[Theorem~2.1]{Gruenberg1957}.
\end{proof}

\begin{proposition}\label{proposition206}
Let $X$ be~a~group, $x$ and~$y$ be~elements of~$X$ such that $x^{-1}yx = y^{-1}$. Let also $p$ be~a~prime number and~$\psi$ be~a~homomorphism of~$X$ onto~a~finite $p$\nobreakdash-group. If~$p \ne 2$, then $y\psi = 1$.
\end{proposition}

\begin{proof}
Let $\gamma_{i}(X)$ denote the~$i$\nobreakdash-th member of~the~lower central series of~$X$, and~let $r$ be~the~order of~$y\psi$. It~is~easy to~verify that $y^{2^{i}} \in \gamma_{i+1}(X)$ for~every $i \geqslant 0$. Since $X\psi$ is~nilpotent, it~follows that $y^{2^{i}}\psi = 1$ for~some $i \geqslant 0$. If~$p \ne 2$, then $1 = (r,2^{i}) = r\alpha + 2^{i}\beta$ for~suitable integers $\alpha$,~$\beta$ and~hence $y\psi = (y\psi)^{r\alpha + 2^{i}\beta} = 1$ as~required.
\end{proof}

\section{The~fundamental group of~a~graph of~groups}\label{section03}

Let $\Gamma$ be~a~non-empty undirected graph with~a~vertex set~$V$ and~an~edge set~$E$ (loops and~multiple edges are allowed). To~turn $\Gamma$ into~a~\emph{graph of~groups}, we denote the~vertices of~$\Gamma$ that are the~ends of~an~edge $e \in E$ by~$e(1)$, $e(-1)$ and~assign to~each vertex $v \in V$ some group~$G_{v}$, to~each edge $e \in E$ a~group~$H_{e}$ and~injective homomorphisms $\varphi_{+e}\colon H_{e} \to G_{e(1)}$, $\varphi_{-e}\colon H_{e} \to G_{e(-1)}$. We denote the~resulting graph of~groups by~$\mathcal{G}(\Gamma)$, the~subgroups $H_{e}\varphi_{+e}$ and~$H_{e}\varphi_{-e}$ ($e \in E$) by~$H_{+e}$ and~$H_{-e}$. We also call $G_{v}$ ($v \in V$), $H_{e}$ ($e \in E$), and~$H_{\varepsilon e}$ ($e \in E$, $\varepsilon = \pm 1$) \emph{vertex groups}, \emph{edge groups}, and~\emph{edge subgroups} respectively. All designations introduced in~this paragraph are assumed to~be fixed until the~end of~the~section.

It should be~noted that an~edge~$e$ of~$\mathcal{G}(\Gamma)$ is~associated with~two different homomorphisms $\varphi_{+e}$, $\varphi_{-e}$ even in~the~case when $e$ is~a~loop, i.\;e.~$e(1) = e(-1)$. Therefore, we can consider $\mathcal{G}(\Gamma)$ as~a~directed graph assuming that $\varphi_{+e}$ corresponds to~the~origin while $\varphi_{-e}$ does to~the~terminus of~$e$.

Let $F$ be~a~maximal forest in~$\Gamma$ and~$E_{F}$ be~the~set of~edges of~$\Gamma$ that belong to~$F$. The~\emph{fundamental group} of~$\mathcal{G}(\Gamma)$ is~the~group~$\pi_{1}(\mathcal{G}(\Gamma))$ whose generators are the~generators of~$G_{v}$ ($v \in V$) and~symbols~$t_{e}$ ($e \in E \setminus E_{F}$) and~whose defining relations are the~relations of~$G_{v}$ ($v \in V$) and~all possible relations of~the~form
\begin{align*}
h_{e}^{\vphantom{1}}\varphi_{+e}^{\vphantom{1}} &= 
h_{e}^{\vphantom{1}}\varphi_{-e}^{\vphantom{1}} \quad 
(e \in E_{F}^{\vphantom{1}},\ 
h_{e}^{\vphantom{1}} \in H_{e}^{\vphantom{1}}),\\[-3pt]
t_{e}^{-1}(h_{e}^{\vphantom{1}}
\varphi_{+e}^{\vphantom{1}})
t_{e}^{\vphantom{1}} &= 
h_{e}^{\vphantom{1}}\varphi_{-e}^{\vphantom{1}} \quad 
(e \in E \setminus E_{F}^{\vphantom{1}},\ 
h_{e}^{\vphantom{1}} \in H_{e}^{\vphantom{1}}),
\end{align*}
where $h_e\varphi_{\varepsilon e}$ ($\varepsilon = \pm 1$) is~the~word in~the~generators of~$G_{e(\varepsilon)}$ defining the~image of~$h_e$ under~$\varphi_{\varepsilon e}$~\cite[\S~5.1]{Serre1980}.

Obviously, the~presentation of~$\pi_{1}(\mathcal{G}(\Gamma))$ depends on~the~choice of~$F$. It~is~known, however, that all the~groups with~the~presentations corresponding to~different maximal forests of~$\Gamma$ are isomorphic~\cite[\S~5.1]{Serre1980}. This allows us to~talk about the~fundamental group of~a~graph of~groups without mentioning a~specific maximal forest. It~is~also known that, for~each vertex $v \in V$, the~identity mapping of~the~generators of~$G_{v}$ to~$\pi_{1}(\mathcal{G}(\Gamma))$ defines an~injective homomorphism~\cite[\S~5.2]{Serre1980} and~so $G_{v}$ can be~considered as~a~subgroup of~$\pi_{1}(\mathcal{G}(\Gamma))$. This easily implies

\begin{proposition}\label{proposition301}
Let $\Gamma^{\prime}$ be~an~arbitrary connected subgraph of~$\Gamma$, $T^{\prime}$ be~a~maximal subtree of~$\Gamma^{\prime}$, and~$\mathcal{G}(\Gamma^{\prime})$ be~the~graph of~groups whose vertices and~edges correspond to~the~same groups and~homomorphisms as~in~$\mathcal{G}(\Gamma)$. Then there exists a~maximal forest~$F$ in~$\Gamma$ such that $F \cap \Gamma^{\prime} = T^{\prime}$. If~the~presentations of~$\pi_{1}(\mathcal{G}(\Gamma))$ and~$\pi_{1}(\mathcal{G}(\Gamma^{\prime}))$ correspond to~the~indicated forest~$F$ and~$T^{\prime}$, then the~identity mapping of~the~generators of~$\pi_{1}(\mathcal{G}(\Gamma^{\prime}))$ to~$\pi_{1}(\mathcal{G}(\Gamma))$ defines an~injective homomorphism.
\end{proposition}

The~next statement is~a~special case of~\cite[Proposition~13]{SokolovTumanova2020LJM}.

\begin{proposition}\label{proposition302}
Let $\Gamma$ be~finite and~$N$ be~a~normal subgroup of~$\pi_{1}(\mathcal{G}(\Gamma))$ that meets each subgroup~$G_{v}$ \textup{(}$v \in V$\textup{)} trivially. Then $N$ is~free.
\end{proposition}

As~usual, we say that a~group possesses some property \emph{locally} if~each of~its finitely generated subgroups possesses this property.

\begin{proposition}\label{proposition303}
\textup{\cite[Theorem~1]{Kuvaev2019}}
Let $\Gamma$ be~connected, every $G_{v}$ \textup{(}$v \in V$\textup{)} locally satisfy a~non-trivi\-al identity, and, for~each $e \in E$, $[G_{e(1)}:H_{+e}] \ne 1 \ne [G_{e(-1)}:H_{-e}]$, $[G_{e(1)}:H_{+e}] \cdot [G_{e(-1)}:H_{-e}] > 4$. If~$\pi_{1}(\mathcal{G}(\Gamma))$ is~locally residually nilpotent, then there exists a~prime number~$p$ such that, for~any $e \in E$, $\varepsilon = \pm 1$, $H_{\varepsilon e}$ is~$p^{\prime}$\nobreakdash-isolated in~$G_{e(\varepsilon)}$ \textup{(}i.\;e.,~for~each $g \in G_{e(\varepsilon)}$ and~for~each prime number~$q$, it~follows from~$g^{q} \in H_{\varepsilon e}$ and~$p \ne q$ that $g \in H_{\varepsilon e}$\textup{)}.
\end{proposition}

Let $\Gamma$ consist of~two vertices and~an~edge~$e$ connecting them. Recall that in~this case $\pi_{1}(\mathcal{G}(\Gamma))$ is~said to~be the~\emph{free product of~$G_{e(1)}$ and~$G_{e(-1)}$ with~$H_{+e}$ and~$H_{-e}$ amalgamated}. The~groups $G_{e(1)}$ and~$G_{e(-1)}$ are called the~\emph{free factors} of~this free product (the~terminology used here and~below and~concerning free products with~amalgamated subgroups and~HNN-ex\-ten\-sions follows the~monographs~\cite{MagnusKarrasSolitar1974, LyndonSchupp1980}). The~presentation of~an~element $g \in \pi_{1}(\mathcal{G}(\Gamma))$ in~the~form $g = g_{1}\ldots g_{n}$, $n \geqslant 1$, is~said to~be~\emph{reduced} if~every multiplier $g_{i}$ belongs to~one of~the~groups $G_{e(1)}$, $G_{e(-1)}$ and~no~two neighboring multipliers $g_{i}$, $g_{i+1}$ lie simultaneously in~$G_{e(1)}$ or~$G_{e(-1)}$. The~number~$n$ is~called the~\emph{length} of~this reduced form. The~normal form theorem for~generalized free products (see, e.~g.~\cite[Theorem~4.4]{MagnusKarrasSolitar1974}) implies that if~an~element $g \in \pi_{1}(\mathcal{G}(\Gamma))$ has at~least one reduced form of~length greater than~$1$, then it~does not~belong to~any of~the~free factors $G_{e(1)}$, $G_{e(-1)}$ and,~in~particular, differs from~$1$.

If $\Gamma$ has only one vertex~$v$ and~at least one loop, then $\pi_{1}(\mathcal{G}(\Gamma))$ is~said to~be the~\emph{HNN-ex\-ten\-sion of~$G_{v}$ with~the~stable letters~$t_{e}$} ($e \in E$). The~group~$G_{v}$ is~called the~\emph{base group} of~this HNN-ex\-ten\-sion. In~this case, by~a~\emph{reduced form} of~an~element $g \in \pi_{1}(\mathcal{G}(\Gamma))$ we mean the~product $g = g_{0}^{\vphantom{\mbox{}_1}}t_{e_{i_1}}^{\varepsilon_1}
g_{1}^{\vphantom{\mbox{}_1}} \ldots 
t_{e_{i_n}}^{\varepsilon_n}g_{n}^{\vphantom{\mbox{}_1}}$, where $n \geqslant 0$, $g_{0}, g_{1}, \ldots, g_{n} \in G_{v}$, $e_{i_1}, \ldots, e_{i_n} \in E$, $\varepsilon_{1}, \ldots, \varepsilon_{n} \in \{1,-1\}$, and,~for~each $k \in \{1, \ldots, n-1\}$, if~$i_{k} = i_{k+1}$ and~$\varepsilon_{k} = - \varepsilon_{k+1}$, then $g_{k} \notin H_{-\varepsilon_{k}e_{i_k}}$. As~above, $n$ is~called the~\emph{length} of~this reduced form. It~is~known~\cite{Britton1963} that if~an~element $g \in \pi_{1}(\mathcal{G}(\Gamma))$ has at~least one reduced form of~length greater than~$0$, then it~does not~belong to~the~base group~$G_{v}$ and,~in~particular, differs from~$1$.

\begin{proposition}\label{proposition304}
Let $P(m,n) = \langle x, y;\ x^{m} = y^{n} \rangle$, $1 < |m|$,\,$|n|$, $\mathcal{C}$ be~an~arbitrary class of~groups consisting only of~periodic groups, and~$\rho(\mathcal{C})$ be~the~set of~primes defined in~Section~\textup{\ref{section01}}. If~$P(m,n)$ is~residually a~$\mathcal{C}$\nobreakdash-group, then $m$ and~$n$ are $\rho(\mathcal{C})$\nobreakdash-num\-bers.
\end{proposition}

\begin{proof}
Suppose that $m$ is~not~a~$\rho(\mathcal{C})$\nobreakdash-num\-ber, i.\;e.~there exists a~prime number $p \notin \rho(\mathcal{C})$ such that $p \mid m$. Let $k = m/p$ and~$z = [x^{k},y]$.

Obviously, $P(m,n)$ is~the~free product of~the~infinite cyclic groups $\langle x \rangle$ and~$\langle y \rangle$ with~the subgroups $\langle x^{m} \rangle$ and~$\langle y^{n} \rangle$ amalgamated. Since $|k| < |m|$ and~$1 < |n|$, then $x^{k} \notin \langle x^{m} \rangle$ and~$y \notin \langle y^{n} \rangle$. Therefore, $z$ has a~reduced form of~length~$4$ and~hence differs from~$1$.

Let $\psi$ be~an~arbitrary homomorphism of~$P(m,n)$ onto~a~$\mathcal{C}$\nobreakdash-group. Then the~order~$q$ of~$x\psi$ is~finite and~is a~$\rho(\mathcal{C})$\nobreakdash-num\-ber. Since $p \notin \rho(\mathcal{C})$, then $1 = (p,q) = \alpha p + \beta q$ for~some integers $\alpha$, $\beta$ and~$x^{k}\psi = (x^{k}\psi)^{\alpha p + \beta q} = (x^{k}\psi)^{\alpha p} = (x^{m}\psi)^{\alpha} = (y^{n}\psi)^{\alpha}$. Therefore, $z\psi = 1$. Since $\psi$ is~chosen arbitrarily, it~follows that $P(m,n)$ is~not~residually a~$\mathcal{C}$\nobreakdash-group.

Similar arguments prove that if~$P(m,n)$ is~residually a~$\mathcal{C}$\nobreakdash-group, then $n$ is~a~$\rho(\mathcal{C})$\nobreakdash-num\-ber.
\end{proof}

If all the~vertex and~edge groups of~$\mathcal{G}(\Gamma)$ are infinite cyclic and~their generators~$g_{v}$ ($v \in V$) and~$h_{e}$ ($e \in E$) are fixed, then the~homomorphism~$\varphi_{\varepsilon e}$ ($e \in E$, $\varepsilon = \pm 1$) is~uniquely defined by~a~number $\lambda(\varepsilon e) \in \mathbb{Z} \setminus \{0\}$ such that $g_{e(\varepsilon)}^{\lambda(\varepsilon e)} = h_{\vphantom{(}e}^{\vphantom{(}}\varphi_{\vphantom{(}\varepsilon e}^{\vphantom{(}}$. Therefore, \pagebreak instead of~$\mathcal{G}(\Gamma)$, we can consider a~\emph{labeled graph}~$\mathcal{L}(\Gamma)$, which is~obtained from~$\Gamma$ by~associating each edge $e \in E$ with~non-zero integers $\lambda(+e)$ and~$\lambda(-e)$.

If all the~vertex and~edge groups of~$\mathcal{G}(\Gamma)$ are finite cyclic, then $\mathcal{G}(\Gamma)$ can be~replaced by~a~graph~$\mathcal{M}(\Gamma)$, in~which labels are assigned not~only to~the~edges, but also to~the~vertices: the~label~$\mu(v)$ at~a~vertex~$v$ means that the~vertex group~$G_{v}$ is~of~order~$\mu(v)$. Of~course, for~each edge $e \in E$, the~equality $|\mu(e(1))/\lambda(+e)| = |\mu(e(-1))/\lambda(-e)|$ must hold. We need such graphs in~our proofs.

We call the~group defined by~$\mathcal{L}(\Gamma)$ $(\mathcal{M}(\Gamma))$ the~\emph{fundamental group of~the~labeled graph} $\mathcal{L}(\Gamma)$ $(\mathcal{M}(\Gamma))$ and~denote it~by~$\pi_{1}(\mathcal{L}(\Gamma))$ (respectively $\pi_{1}(\mathcal{M}(\Gamma))$). In~order to~avoid ambiguity when specifying the~presentation of~this group, $\mathcal{L}(\Gamma)$ (and~$\mathcal{M}(\Gamma)$) must be~considered directed. In~each of~these graphs, the~ends of~an~edge~$e$ are, as~before, denoted by~$e(1)$ and~$e(-1)$.

\section{GBS-groups and~their properties}\label{section04}

It follows from~the~previous section that each GBS-group can be defined by~a~labeled graph~$\mathcal{L}(\Gamma)$ for~some finite connected graph~$\Gamma$ and~vice versa, each labeled graph~$\mathcal{L}(\Gamma)$ over a~non-empty finite connected graph~$\Gamma$ defines some GBS-group. Until the~end of~the~paper, we assume that $\Gamma = (V,E)$ is~an~arbitrary non-empty finite connected graph with~a~vertex set~$V$ and~an~edge set~$E$, $\mathcal{L}(\Gamma)$ is~a~graph with~labels~$\lambda(\varepsilon e)$ ($e \in E$, $\varepsilon = \pm 1$), and~$G$ is~the~corresponding GBS-group with~the~vertex groups $G_{v} = \langle g_{v} \rangle$ ($v \in V$) and~the~edge subgroups $H_{\vphantom{(}\varepsilon e}^{\vphantom{(}} = \big\langle g_{e(\varepsilon)}^{\lambda(\varepsilon e)} \big\rangle$ ($e \in E$, $\varepsilon = \pm 1$). If~$\Gamma^{\prime}$ is~a~subgraph of~$\Gamma$, then by~$\mathcal{L}(\Gamma^{\prime})$ we denote the~labeled graph, the~edges of~which are associated with~the~same labels as~in~$\mathcal{L}(\Gamma)$.

As~mentioned above, the~graph~$\mathcal{L}(\Gamma)$ is~called \emph{reduced} if, for~each $e \in E$, $\varepsilon = \pm 1$, the~equality $|\lambda(\varepsilon e)| = 1$ implies that $e$ is~a~loop~\cite[p.~224]{Forester2002}. Suppose that $\mathcal{L}(\Gamma)$ is~not~reduced. Then it~contains an~edge~$e$ such that $e(1) \ne e(-1)$ and~$|\lambda(\varepsilon e)| = 1$ for~some $\varepsilon = \pm 1$. Let us choose a~maximal subtree of~$\Gamma$ containing~$e$. Then $g_{\vphantom{(}e(\varepsilon)}^{\vphantom{(}} = g_{e(-\varepsilon)}^{\lambda(\varepsilon e)\lambda(-\varepsilon e)}$ in~$G$ and~hence the~generator~$g_{e(\varepsilon)}$ can be~excluded from~the~presentation of~$G$. In~$\mathcal{L}(\Gamma)$, this operation corresponds to~the~contraction of~$e$ with~preliminary multiplication of~all the~labels around the~vertex $e(\varepsilon)$ by~$\lambda(\varepsilon e)\lambda(-\varepsilon e)$. Such a~transformation of~$\mathcal{L}(\Gamma)$ is~called an~\emph{elementary collapse} (see~\cite[p.~480]{Levitt2007}). Since $\Gamma$ is~finite, then $\mathcal{L}(\Gamma)$ can always be~reduced by~performing a~finite number of~elementary collapses.

If we replace the~generator of~a~certain vertex group with~its inverse, then all the~labels around the~corresponding vertex change sign. Similarly, replacing the~generator of~a~certain edge group with~the~inverse leads to~a~change in~the~signs of~the~labels at~the~ends of~this edge. The~listed changes of~the~generators induce isomorphisms of~$G$, and~the~corresponding graph transformations are called \emph{admissible changes of~signs}~\cite[p.~479]{Levitt2007}.

Let some maximal subtree~$T$ of~$\Gamma$ be~fixed. It~is~easy to~see that one can make all the~labels at~the~ends of~the~edges of~$T$ positive by~applying suitable admissible sign changes. We call the~resulting graph~$\mathcal{L}(\Gamma)$ \emph{$T$\nobreakdash-posi\-tive}.

An~element $a \in G$ is~said to~be~\emph{elliptic} if~it~is~conjugate to~an~element of~some vertex group. If~$G$ is~not~elementary, then the~ellipticity of~an~element does not~depend on~the~choice of~the~graph $\mathcal{L}(\Gamma)$ defining~$G$, the~set of~elliptic elements is~invariant under~automorphisms of~$G$, and~any two elliptic elements $a$,~$b \in G$ are commensurable, i.\;e.~$\langle a \rangle \cap \langle b \rangle \ne 1$~\cite[Lemma~2.1, Corollary~2.2]{Levitt2007}. This allows us to~define the~mapping $\Delta\colon G \to \mathbb{Q}^{*}$ as~follows.

Let $g \in G$ be~an~arbitrary element. Take a~non-unit elliptic element~$a$. Then the~element~$g^{-1}ag$ is~also elliptic and~hence there exist numbers $m$ and~$n$ such that $g^{-1}a^{m}g = a^{n}$. We put $\Delta(g) = n/m$.

This definition does not~depend on~the~choice of~$a$, $m$, and~$n$~\cite{Kropholler1990}. The~constructed mapping~$\Delta$ is~called the~\emph{modular homomorphism} of~$G$. The~notation~$\Delta$ is~used below without special explanations.

\begin{proposition}\label{proposition401}
\textup{\cite[Propositions~7.5,~7.11]{Levitt2015}}
Let $G$ be~non-solv\-a\-ble and $n/m \in \operatorname{Im}\Delta \setminus \{1\}$ be~a~rational number written in~lowest terms. Then $G$ contains a~subgroup isomorphic to~$\mathrm{BS}(m,n)$.
\end{proposition}

\begin{proposition}\label{proposition402}
\textup{\cite[Lemma~7.6]{Levitt2015}}
If $G$ is~non-solv\-a\-ble and~contains a~subgroup isomorphic to~$\mathrm{BS}(1,n)$, $|n| \ne 1$, then it~contains a~subgroup isomorphic to~$\mathrm{BS}(q,qn)$, where $q$ is~some prime number.
\end{proposition}

\begin{proposition}\label{proposition403}
Let $\Gamma$ be~a~tree and~$\mathcal{I}$ be~a~non-empty finite set of~indices, which is~the~disjoint union of~the~set $\big\{(e,\varepsilon) \mid e \in E,\ \varepsilon = \pm 1\big\}$ and~some set~$\mathcal{J}$. Let also $\Sigma = \{H_{i} \mid i \in \mathcal{I}\}$ be~a~family of~subgroups of~$G_{v}$ \textup{(}$v \in V$\textup{)} and~$\nu\colon \mathcal{I} \to V$ be~a~function such that, for~any $i \in \mathcal{I}$, $1 \ne H_{i} \leqslant G_{\nu(i)}$ and~if $i = (e,\varepsilon)$ for~some $e \in E$, $\varepsilon = \pm 1$, then $H_{i} = H_{e\varepsilon}$ and~$\nu(i) = e(\varepsilon)$. Finally, let $K = \bigcap_{i \in \mathcal{I}} H_{i}$ and~$\chi(i) = [G_{\nu(i)}:H_{i}]$. Then the~following statements hold.

\textup{1.}\hspace{1ex}$K \leqslant \bigcap_{v \in V} G_{v}$ and~therefore the~numbers $\mu(v) = [G_{v}:K]$ \textup{(}$v \in V$\textup{)} are defined.

\textup{2.}\hspace{1ex}$K \ne 1$ and~therefore all the~numbers~$\mu(v)$ \textup{(}$v \in V$\textup{)} are finite.

\textup{3.}\hspace{1ex}The least common multiple~$\mu$ of~$\mu(v)$ \textup{(}$v \in V$\textup{)} divides $\prod_{i \in \mathcal{I}} \chi(i)$.
\end{proposition}

\begin{proof}
\textup{1.}\hspace{1ex}We note that $V = \{\nu(i) \mid i \in \mathcal{I}\}$. Indeed, if~$\Gamma$ consists of~one vertex~$v$, then $v = \nu(i)$ for~all $i \in \mathcal{I}$ and~the~desired equality holds because $\mathcal{I}$ is~non-empty. Otherwise, each vertex is~incident to~at~least one edge, and~so, for~any $v \in V$, there exist $e \in E$, $\varepsilon = \pm 1$ such that $v = e(\varepsilon) = \nu(i)$, where $i = (e,\varepsilon) \in \mathcal{I}$. Hence,
$$
K = \bigcap_{i \in \mathcal{I}} H_{i} \leqslant \bigcap_{i \in \mathcal{I}} G_{\nu(i)} = \bigcap_{v \in V} G_{v},
$$
as~required.

\textup{2.}\hspace{1ex}We put $H = \bigcap_{v \in V} G_{v}$ and~use induction on~the~number of~vertices in~$\Gamma$ to~show that $H$ is~an~infinite cyclic subgroup. If~$\Gamma$ contains only one vertex $v$, then $H = G_{v}$ and~the~required statement is~obvious. Therefore, we further assume that $\Gamma$ contains more than one vertex and~so $E \ne \varnothing$.

{\parfillskip=0pt
Let $f \in E$ be~an~arbitrary edge and~$\Gamma - f$ be~the~graph that is~obtained from~$\Gamma$ by~removing~$f$. Since $\Gamma$ is~a~tree, $\Gamma - f$ has exactly two connected components. For~every $\varepsilon = \pm 1$, denote by~$\Gamma_{\varepsilon}$ the~connected component of~$\Gamma - f$ that contains $f(\varepsilon)$ and~by~$V_{\varepsilon}$ the~vertex set of~$\Gamma_{\varepsilon}$. By~the~inductive hypothesis, the~subgroup $H_{\varepsilon} = \bigcap_{v \in V_{\varepsilon}} G_{v}$ is~infinite cyclic, and~hence $H_{\varepsilon} \cap H_{\varepsilon f} \ne 1$ as~the~intersection of~two non-trivi\-al subgroups of~$G_{f(\varepsilon)}$.\par}

By~Proposition~\ref{proposition301}, the~free product of~the~groups $G_{f(1)}$ and~$G_{f(-1)}$ with~the~subgroups $H_{+f}$ and~$H_{-f}$ amalgamated is~embedded into~$G$ by~means of~the~identity mapping of~the~generators. Therefore, the~equalities $H_{+f} = G_{f(1)} \cap G_{f(-1)} = H_{-f}$ hold in~$G$~\cite[Theorem~4.4.3]{MagnusKarrasSolitar1974},~and
$$
H = H_{1} \cap H_{-1} = (H_{1} \cap G_{f(1)}) \cap (H_{-1} \cap G_{f(-1)}) = (H_{1} \cap H_{+f}) \cap (H_{-1} \cap H_{-f})
$$
{\parfillskip=0pt
is~the~intersection of~two non-trivi\-al subgroups of~the~infinite cyclic group $H_{+f} = H_{-f}$.\par}

Thus, $H \ne 1$, and~so $H_{i} \cap H$ is~an~infinite cyclic subgroup of~$G_{\nu(i)}$ for~any $i \in \mathcal{I}$. As~noted above, $V = \{\nu(i) \mid i \in \mathcal{I}\}$, hence $K = \bigcap_{i \in \mathcal{I}} (H_{i} \cap G_{\nu(i)}) = \bigcap_{i \in \mathcal{I}} (H_{i} \cap H)$. Since all the~subgroups $H_{i} \cap H$ ($i \in \mathcal{I}$) lie in~$H$ and~$\mathcal{I}$ is~finite, it~follows that $K \ne 1$.

\textup{3.}\hspace{1ex}We again use induction on~the~number of~vertices in~$\Gamma$. If~$\Gamma$ contains only one vertex~$v$, then $\mu = \mu(v) = [G_{v}:K]$ and~$\chi(i) = [G_{v}:H_{i}]$. Therefore, the~required statement follows from~the~relation $[G_{v}:\bigcap_{i \in \mathcal{I}} H_{i}] \mid \prod_{i \in \mathcal{I}}[G_{v}:H_{i}]$, and~we further assume that $\Gamma$ has at~least two vertices.

Let us choose an~arbitrary edge $f \in E$ and~denote by~$\Gamma_{\varepsilon}$ ($\varepsilon = \pm 1$) the~connected component of~the~graph $\Gamma - f$ that contains~$f(\varepsilon)$. Let also $V_{\varepsilon}$ be~the~vertex set of~$\Gamma_{\varepsilon}$, $\mathcal{I}_{\varepsilon} = \{i \mid i \in \mathcal{I},\ \nu(i) \in V_{\varepsilon}\}$, $K_{\varepsilon} = \bigcap_{i \in \mathcal{I}_{\varepsilon}} H_{i}$, and~$p_{\varepsilon} = \prod_{i \in \mathcal{I}_{\varepsilon}} \chi(i)$. Then $\prod_{i \in \mathcal{I}} \chi(i) = p_{1}p_{-1}$ and~the~equality $K = K_{1} \cap K_{-1}$ holds in~$G$.

By~the~definition of~$\nu$, the~set of~indices $\mathcal{I}_{\varepsilon}$ contains the~pair $(f,\varepsilon)$ and~is therefore non-empty. It~is~also easy to~see that $\Gamma_{\varepsilon}$, $\mathcal{I}_{\varepsilon}$, $\Sigma_{\varepsilon} = \{H_{i} \mid i \in \mathcal{I}_{\varepsilon}\}$, and~$\nu_{\varepsilon} = \nu|_{\mathcal{I}_{\varepsilon}}$ satisfy all the~conditions of~the~proposition. Hence, the~numbers $\mu_{\varepsilon}(v) = [G_{v}:K_{\varepsilon}]$ ($v \in V_{\varepsilon}$) are defined and~finite in~view of~Statements~1,~2, while their least common multiple~$\mu_{\varepsilon}$ divides~$p_{\varepsilon}$ by~the~inductive hypothesis.

Since $K_{1} \leqslant H_{+f}$, $K_{-1} \leqslant H_{-f}$, and~$H_{+f} = G_{f(1)} \cap G_{f(-1)} = H_{-f}$, then $K_{1}K_{-1} \leqslant G_{f(1)} \cap G_{f(-1)}$ and~$[K_{1}K_{-1}:K_{\varepsilon}] \mid [G_{f(\varepsilon)}:K_{\varepsilon}] = \mu_{\varepsilon}(f(\varepsilon)) \mid \mu_{\varepsilon}$ for~every $\varepsilon = \pm 1$. Therefore, for~any $\varepsilon = \pm 1$, $v \in V_{\varepsilon}$, we have
$$
\mu(v) = [G_{v}:K] = [G_{v}:K_{\varepsilon}][K_{\varepsilon}:K_{\varepsilon} \cap K_{-\varepsilon}] = [G_{v}:K_{\varepsilon}][K_{\varepsilon}K_{-\varepsilon}:K_{-\varepsilon}] \mid \mu_{\varepsilon}\mu_{-\varepsilon} \mid p_{1}p_{-1}.
$$
Hence, the~least common multiple of~the~numbers~$\mu(v)$ ($v \in V = V_{1} \cup V_{-1}$) also divides the~product~$p_{1}p_{-1}$.
\end{proof}

\begin{proposition}\label{proposition404}
Let $G$ be~non-el\-e\-men\-tary, $T$ be~a~maximal subtree of~$\Gamma$, $E_{T}$ be~the~edge set of~$T$, and~$K = \bigcap_{e \in E,\,\varepsilon = \pm 1} H_{\varepsilon e}$. Then the~following statements hold.

\textup{1.}\hspace{1ex}$K \leqslant \bigcap_{v \in V} G_{v}$ and~therefore the~numbers $\mu(v) = [G_{v}:K]$ \textup{(}$v \in V$\textup{)} are defined.

\textup{2.}\hspace{1ex}$K \ne 1$ and~therefore all the~numbers~$\mu(v)$ \textup{(}$v \in V$\textup{)} are finite.

\textup{3.}\hspace{1ex}The~least common multiple~$\mu$ of~$\mu(v)$ \textup{(}$v \in V$\textup{)} divides $\prod_{e \in E,\,\varepsilon = \pm 1} \lambda(\varepsilon e)$.

\textup{4.}\hspace{1ex}If $\mathcal{L}(\Gamma)$ is~$T$\nobreakdash-posi\-tive, then $g_{v}^{\mu(v)} = g_{w}^{\mu(w)}$ for~any $v,w \in V$ and~$\lambda(+e)/\mu(e(1)) = \lambda(-e)/\mu(e(-1))$ for~any $e \in E_{T}$.

\textup{5.}\hspace{1ex}If $\operatorname{Im}\Delta \subseteq \{1,-1\}$, then $K$ is~normal in~$G$ and~the~centralizer of~$K$ in~$G$ coincides with~$\Delta^{-1}(1)$.

\textup{6.}\hspace{1ex}If $\operatorname{Im}\Delta \subseteq \{1,-1\}$ and~$\tau$ is~a~homomorphism of~$G$ such that $\ker\tau \cap G_{v} = K$ for~all $v \in V$, then $\ker\tau$ is~an~extension of~$K$ by~a~free group.

\textup{7.}\hspace{1ex}If $\operatorname{Im}\Delta \subseteq \{1,-1\}$ and~$\mathcal{L}(\Gamma)$ is~reduced, then $C(G) = K$.
\end{proposition}

\begin{proof}
1,\hspace{.9ex}2,\hspace{.9ex}3.\hspace{1ex}Since $G$ is~an~HNN-ex\-ten\-sion of~the~tree product $P = \pi_{1}(\mathcal{L}(T))$ and~all the~groups~$G_{v}$ ($v \in V$) are contained in~$P$, then Statements~1\nobreakdash--3 are obtained by~applying Proposition~\ref{proposition403} to~the~tree~$T$, the~group~$P$, the~set
$$
\mathcal{I} = 
\big\{(e,\varepsilon) \mid e \in E_{T},\,\varepsilon = \pm 1\big\} \cup 
\big\{(e,\varepsilon) \mid e \in E \setminus E_{T},\,\varepsilon = \pm 1\big\},
$$
the~family $\Sigma = \{H_{e\varepsilon} \mid e \in E,\ \varepsilon = \pm 1\}$, and~the~function $\nu\colon \mathcal{I} \to V$ such that $\nu(e,\varepsilon) = e(\varepsilon)$. It~should only be~noted that $E$ and~$\mathcal{I}$ are non-empty because $G$ is~not~elementary. Besides, the~relation $\mu \mid \prod_{e \in E,\,\varepsilon = \pm 1} |\lambda(\varepsilon e)|$, which follows from~Statement~3 of~Proposition~\ref{proposition403}, is~equivalent to~the~required~one.

4.\hspace{1ex}We use induction on~the~length of~the~path connecting $v$ and~$w$ in~$T$. If~$v = w$, the~statement is~obvious, so we assume that this path contains an~edge~$e$ connecting~$v$ with~some vertex~$u$ (which may coincide with~$w$). Let also, for~definiteness, $e(1) = v$ and~$e(-1) = u$.

\vspace{-1pt}

By~the~inductive hypothesis, $g_{u}^{\mu(u)} = g_{w}^{\mu(w)}$. Since $e \in E_{T}$, the~equalities $g_{v}^{\lambda(+e)} = g_{u}^{\lambda(-e)}$ and~$H_{+e} = H_{-e}$ hold in~$G$. It~follows that $[H_{+e}:K] = [H_{-e}:K] = k$ for~some $k \geqslant 1$. Since \hfill $\mathcal{L}(\Gamma)$ \hfill is \hfill $T$\nobreakdash-posi\-tive, \hfill then \hfill $[G_{v}:H_{+e}] = \lambda(+e)$ \hfill and \hfill $[G_{u}:H_{-e}] = \lambda(-e)$. Therefore,
\begin{align*}
\mu(v) = [G_{v}:K] &= [G_{v}:H_{+e}][H_{+e}:K] = \lambda(+e)k,\\
\mu(u) = [G_{u}:K] &= [G_{u}:H_{-e}][H_{-e}:K] = \lambda(-e)k,
\end{align*}
and
$$
g_{v}^{\mu(v)} = g_{v}^{\lambda(+e)k} = g_{u}^{\lambda(-e)k} = g_{u}^{\mu(u)} = g_{w}^{\mu(w)}.
$$
Since the~vertices $v$ and~$w$ are chosen arbitrarily, the~above equalities also imply that $\lambda(+e)/\mu(e(1)) = \lambda(-e)/\mu(e(-1))$ for~any $e \in E_{T}$.

{\parfillskip=0pt
5.\hspace{1ex}Let $x$ be~a~generator of~$K$. It~is~obvious that $[x,g_{v}] = 1$ for~every $v \in V$. The~element~$x$ is~elliptic, so if~$e \in E \setminus E_{T}$, then there exists a~number $n \geqslant 1$ such that $t_{e}^{-1}x^{n}t_{e}^{\vphantom{1}} = x^{\Delta(t_e)n}$. Since $\operatorname{Im}\Delta \subseteq \{1,-1\}$, $t_{e}^{-1}xt_{e}^{\vphantom{1}} \in H_{-e}$, $x^{\Delta(t_e)} \in H_{-e}$, and~$H_{-e}$ is~infinite cyclic, the~last equality implies that $t_{e}^{-1}xt_{e}^{\vphantom{1}} = x^{\Delta(t_e)}$. Therefore, $K$ is~normal in~$G$.\par}

If $g \in G$ is~an~arbitrary element, then $g^{-1}xg \in K$ because $K$ is~normal. It~follows that $g^{-1}xg = x^{\Delta(g)}$, and~hence $[g,x] = 1$ if~and~only if~$\Delta(g) = 1$.

6.\hspace{1ex}Consider the~quotient group $\overline{G} = G/K$ and~the~labeled graph~$\mathcal{M}(\Gamma)$ that is~obtained from~$\mathcal{L}(\Gamma)$ by~assigning to~each $v \in V$ the~label $\mu(v) = [G_{v}:K]$. It~is~easy to~see that $\overline{G}$~is~isomorphic to~$\pi_1(\mathcal{M}(\Gamma))$ and~the~vertex groups under~this isomorphism correspond to~the~quotient groups~$G_{v}/K$ ($v \in V$). Since $K \leqslant \ker\tau$, the~mapping $\bar\tau\colon \overline{G} \to \operatorname{Im}\tau$ taking the~coset~$gK$ ($g \in G$) to~$g\tau$ is~well defined and~is a~surjective homomorphism. It~follows from~the~equalities $\ker\tau \cap G_{v} = K$ ($v \in V$) that $\ker\bar\tau \cap G_{v}/K = 1$ for~all $v \in V$. Therefore, by~Proposition~\ref{proposition302}, $\ker\bar\tau$ is~a~free group. It~also easily follows from~the~definition of~$\bar\tau$ that the~preimage of~$\ker\bar\tau$ under~the~natural homomorphism $G \to \overline{G}$ coincides with~$\ker\tau$. Thus, $\ker\tau$ is~an~extension of~$K$ by~the~free group~$\ker\bar\tau$.

7.\hspace{1ex}Let $H = \bigcap_{v \in V} G_{v}$. By~\cite[Lemma~5]{DelgadoRobinsonTimm2017}, $C(G) \leqslant H$. Let us show that, for~every element $h \in H \setminus K$, there exists an~element $g \in G$ such that $g^{-1}hg \notin H$. This will mean that $K$ is~the~largest subgroup of~$H$ normal in~$G$ and~therefore $K = C(G)$.

Let $h \in H \setminus K$ be~an~arbitrary element. Then $h \notin H_{\varepsilon e}$ for~some $e \in E \setminus E_{T}$, $\varepsilon = \pm 1$. Indeed, it~is~obvious if~$\Gamma$ contains only one vertex and~so $E_{T} = \varnothing$. Suppose that $E_{T} \ne \varnothing$. Then every vertex of~$\Gamma$ is~incident to~some $e \in E_{T}$. For~any $e \in E_{T}$, the~equalities $H_{+e} = G_{e(1)} \cap G_{e(-1)} = H_{-e}$ hold in~$G$, as~already noted in~the~proof of~Proposition~\ref{proposition403}. Therefore, $H = \bigcap_{e \in E_T,\,\varepsilon = \pm 1} H_{\varepsilon e}$, and~$h$ possesses the~desired property.

Let us consider $G$ as~an~HNN-ex\-ten\-sion with~the~stable letter~$t_{e}$. Then the~element $t_{e}^{-\varepsilon}ht_{e}^{\varepsilon}$ has a~reduced form of~length~$2$ and~hence cannot belong to~$H$ contained in~the~base group of~this HNN-ex\-ten\-sion. Thus, $t_{e}^{\varepsilon}$ is~the~required element.
\end{proof}

\section{Proofs of~Theorems~\ref{theorem03},~\ref{theorem04} and~Corollaries~\ref{corollary01}--\ref{corollary03}}\label{section05}

\begin{proposition}\label{proposition501}
Let $G$ be~non-el\-e\-men\-tary, $T$ be~a~maximal subtree of~$\Gamma$, and~$\mathcal{L}(\Gamma)$ be~$T$\nobreakdash-posi\-tive. Let also $K = \bigcap_{e \in E,\,\varepsilon = \pm 1} H_{\varepsilon e}$ and~$\mu$ be~the~least common multiple of~$\mu(v) = [G_{v}:K]$ \textup{(}$v \in V$\textup{)}. Finally, let $Q$ be~the~subring of~$\mathbb{Q}$ generated by~$\operatorname{Im}\Delta$, $Q^{+}$ be~the~additive group of~$Q$, $A$ be~a~free abelian group with~the~basis $\{a_{q} \mid q \in \operatorname{Im}\Delta\}$, and~$X$ be~the~splitting extension of~$Q^{+}$ by~$A$ such that the~automorphism $\widehat{a_{q}}|_{Q^{+}}$ acts as~multiplication by~$q$. Then the~mapping of~the~generators of~$G$ to~$X$ given by~the~rule
$$
g_{v} \mapsto \mu/\mu(v) \quad (v \in V), \quad\quad
t_{e} \mapsto a_{\Delta(t_{e})} \quad (e \in E \setminus E_{T}),
$$
defines a~homomorphism of~$G$ into~$X$.
\end{proposition}

\begin{proof}
We extend the~indicated mapping of~the~generators to~the~mapping of~words~$\sigma$ and~show that the~latter takes all the~defining relations of~$G$ into~the~equalities valid in~$X$.

If $e$ is~an~edge of~$T$, then $\lambda(+e)/\mu(e(1)) = \lambda(-e)/\mu(e(-1))$ by~Proposition~\ref{proposition404}~and
$$
g_{e(1)}^{\lambda(+e)}\sigma = \lambda(+e)\mu/\mu(e(1)) = \lambda(-e)\mu/\mu(e(-1)) = g_{e(-1)}^{\lambda(-e)}\sigma.
$$

Let $e \in E$ be~an~edge that does not~belong to~$T$. By~Proposition~\ref{proposition404}, the~equality $g_{e(1)}^{\mu(e(1))} = g_{e(-1)}^{\mu(e(-1))}$ holds in~$G$; we denote this element by~$g$ for~brevity. Since $g$ is~elliptic~and
$$
t_{\vphantom{(}e}^{\vphantom{(}-1}g^{\vphantom{(}\lambda(+e)\mu/\mu(e(1))}_{\vphantom{(}}t_{\vphantom{(}e}^{\vphantom{(}} = t_{\vphantom{(}e}^{\vphantom{(}-1}g_{e(1)}^{\lambda(+e)\mu}t_{\vphantom{(}e}^{\vphantom{(}} = g_{e(-1)}^{\lambda(-e)\mu} = g^{\vphantom{(}\lambda(-e)\mu/\mu(e(-1))}_{\vphantom{(}},
$$
then $\Delta(t_{e}) = \big(\lambda(-e)\mu/\mu(e(-1))\big) \big/ \big(\lambda(+e)\mu/\mu(e(1))\big)$. This implies that
$$
\left(t_{\vphantom{(}e}^{\vphantom{(}-1}g_{e(1)}^{\lambda(+e)}t_{\vphantom{(}e}^{\vphantom{(}}\right)\kern-2pt\sigma = \big(\lambda(+e)\mu/\mu(e(1))\big) \cdot \Delta(t_{\vphantom{(}e}^{\vphantom{(}}) = \lambda(-e)\mu/\mu(e(-1)) = \left(g_{e(-1)}^{\lambda(-e)}\right)\kern-2pt\sigma.
$$

\vspace{-4.4ex}

\mbox{}
\end{proof}

\vspace{0pt}

\begin{proposition}\label{proposition502}
Let $G$ be~non-el\-e\-men\-tary, $T$ be~a~maximal subtree of~$\Gamma$, and~$\mathcal{L}(\Gamma)$ be~$T$\nobreakdash-posi\-tive. Let also $K = \bigcap_{e \in E,\,\varepsilon = \pm 1} H_{\varepsilon e}$ and~$\mu$ be~the~least common multiple of~$\mu(v) = [G_{v}:K]$ \textup{(}$v \in V$\textup{)}. If~$\operatorname{Im}\Delta = \{1\}$, then $G$ is~an~$(F \times \mathbb{Z})$\nobreakdash-by\nobreakdash-$\mathbb{Z}_{\mu}$\nobreakdash-group, where $F$ is~some~free group \textup{(}i.\;e.~$G$ is~isomorphic to~an~extension of~$F \times \mathbb{Z}$ by~$\mathbb{Z}_{\mu}$\textup{)}. If~$\operatorname{Im}\Delta = \{1,-1\}$, then $G$ is~an~$((F \times \mathbb{Z})$\nobreakdash-by\nobreakdash-$\mathbb{Z}_{\mu})$\nobreakdash-by\nobreakdash-$\mathbb{Z}_{2}$\nobreakdash-group.
\end{proposition}

\begin{proof}
Let $Q$, $X$, and~$\sigma\colon G \to X$ be~the~subring, the~group, and~the~homomorphism from~Proposition~\ref{proposition501}, $E_{T}$ be~the~edge set of~$T$. Since $\operatorname{Im}\Delta \subseteq \{1,-1\}$, then $Q = \mathbb{Z}$. Therefore, $X$ has the~presentation $\big\langle x, a_{1};\ [x, a_{1}] = 1 \big\rangle$ if~$\operatorname{Im}\Delta = \{1\}$,~or
$$
\big\langle x, a_{1}^{\vphantom{1}}, a_{-1}^{\vphantom{1}};\ 
[x, a_{1}^{\vphantom{1}}] = [a_{1}^{\vphantom{1}}, a_{-1}^{\vphantom{1}}] = 1,\ 
a_{-1}^{-1}xa_{-1}^{\vphantom{1}} = x^{-1} \big\rangle
$$
if~$\operatorname{Im}\Delta = \{1,-1\}$ (here $x$ denotes the~generator of~the~additive group~$Q^{+}$ of~$Q$ equal to~$1$).

Let $Y$ be~the~group with~the~presentation $\big\langle x;\ x^{\mu} = 1 \big\rangle$ if~$\operatorname{Im}\Delta = \{1\}$,~and
$$
\big\langle x, a_{-1}^{\vphantom{1}};\ x^{\mu} = 1,\ 
a_{-1}^{2} = 1,\ a_{-1}^{-1}xa_{-1}^{\vphantom{1}} = x^{-1} \big\rangle
$$
if~$\operatorname{Im}\Delta = \{1,-1\}$. Obviously, $\sigma$ can be~extended to~a~homomorphism~$\tau$ of~$G$ into~$Y$. Since $\mu$ is~the~least common multiple of~$\mu(v)$ ($v \in V$), the~greatest common divisor of~$\mu/\mu(v)$ ($v \in V$) is~equal to~$1$ and~therefore $x \in \operatorname{Im}\tau$. If~$\Delta(t_{e}) = 1$ for~each edge $e \in E \setminus E_{T}$, then $\operatorname{Im}\Delta = \{1\}$. Hence, if~$\operatorname{Im}\Delta = \{1,-1\}$, then there exists an~edge $e \in E \setminus E_{T}$ such that $\Delta(t_{e}) = - 1$ and~so $t_{e}\tau = a_{-1}$. Therefore, $\operatorname{Im}\tau = Y$.

Since $G_{v}\tau = \langle x^{\mu/\mu(v)} \rangle$, then $\ker\tau \cap G_{v} = G_{v}^{\mu(v)} = K$ for~every $v \in V$ and,~by~Proposition~\ref{proposition404}, $\ker\tau$ is~an~extension of~$K$ by~a~free group. It~is~well known that such an~extension is~splittable, i.\;e.~$\ker\tau = KF$, where $F$ is~a~free subgroup of~$G$ and~$K \cap F = 1$. It~remains to~show that $[K,F] = 1$ and~therefore $\ker\tau = K \times F$.

If $\operatorname{Im}\Delta = \{1\}$, then $K$ is~central in~$G$ by~Proposition~\ref{proposition404}. Let $\operatorname{Im}\Delta = \{1,-1\}$ and\linebreak $g \in \ker\tau$ be~an~arbitrary element. Because $\tau$ takes $t_{e}$ ($e \in E \setminus E_{T}$, $\Delta(t_{e}) = - 1$) to~$a_{-1}$ and~$g_{v}$ ($v \in V$), $t_{e}$ ($e \in E \setminus E_{T}$, $\Delta(t_{e}) = 1$) into~$\langle x \rangle$, the~number of~occurrences of~the~generators of~the~first type and~inverse to~them in~the~record of~$g$ must be~even. Since the~conjugation of~$K$ by~$t_{e}$ ($e \in E \setminus E_{T}$, $\Delta(t_{e}) = - 1$) is~an~automorphism of~this subgroup of~order~$2$ and~all the~elements~$g_{v}$ ($v \in V$), $t_{e}$ ($e \in E \setminus E_{T}$, $\Delta(t_{e}) = 1$) belong to~the~centralizer of~$K$, then $g$ also belongs to~the~centralizer of~$K$. Hence, $K$ is~central in~$\ker\tau$, and~$[K,F] = 1$.
\end{proof}

\begin{proof}[\textup{\textbf{Proof of~Theorem~\ref{theorem03}.}}]
According to~Proposition~\ref{proposition203}, Theorem~\ref{theorem01} is~valid for~the~class~$\mathcal{C}$ from~the~statement of~Theorem~\ref{theorem03}. Therefore, we can further use it~in~the~process of~proof.

1,\hspace{.9ex}2.\hspace{1ex}\textit{Necessity.} Let $e \in E$ be~an~arbitrary edge. If~$e$ is~not~a~loop, then, by~Proposition~\ref{proposition301}, $G$ contains a~subgroup isomorphic~to
$$
P\big(\lambda(+e),\lambda(-e)\big) = \left\langle 
g_{e(1)}^{\phantom{\lambda()}}, g_{e(-1)}^{\phantom{\lambda()}};\ 
g_{e(1)}^{\lambda(+e)} = g_{e(-1)}^{\lambda(-e)} \right\rangle.
$$
Since the~labeled graph defining~$G$ is~reduced, then $1 < |\lambda(+e)|, |\lambda(-e)|$. Hence, it~follows from~Proposition~\ref{proposition304} that $\lambda(+e)$ and~$\lambda(-e)$ are $\rho(\mathcal{C})$\nobreakdash-num\-bers. If~$e$ is~a~loop, then, again by~Proposition~\ref{proposition301}, $G$ contains a~subgroup isomorphic to~$\mathrm{BS}(\lambda(+e),\lambda(-e))$. Because $\operatorname{Im}\Delta \subseteq \{1,-1\}$, the~equality $|\lambda(+e)| = |\lambda(-e)|$ holds. Therefore, $\lambda(+e)$ and~$\lambda(-e)$ are $\rho(\mathcal{C})$\nobreakdash-num\-bers as~Theorem~\ref{theorem01} states.

Let $\operatorname{Im}\Delta = \{1,-1\}$. Then, by~Proposition~\ref{proposition401}, $G$ contains a~subgroup isomorphic to~$\mathrm{BS}(1,-1)$, and~Theorem~\ref{theorem01} implies that $2 \in \rho(\mathcal{C})$.

\textit{Sufficiency.} Choose some maximal subtree~$T$ in~$\Gamma$ and~transform the~labeled graph~$\mathcal{L}(\Gamma)$ defining~$G$ to~a~$T$\nobreakdash-posi\-tive form. Obviously, all the~labels are still $\rho(\mathcal{C})$\nobreakdash-num\-bers after this operation. Let $K = \bigcap_{e \in E,\,\varepsilon = \pm 1} H_{\varepsilon e}$ and~$\mu$ be~the~least common multiple of~$\mu(v) = [G_{v}:K]$ ($v \in V$). By~Proposition~\ref{proposition404}, $\mu$ is~a~$\rho(\mathcal{C})$\nobreakdash-num\-ber, and,~by~Proposition~\ref{proposition201}, $\mathbb{Z}_{\mu} \in \mathcal{C}$. If~$2 \in \rho(\mathcal{C})$, then, according to~the~same proposition, $\mathbb{Z}_{2} \in \mathcal{C}$. Proposition~\ref{proposition502} states that, for~some free group~$F$, $G$ is~an~$(F \times \mathbb{Z})$\nobreakdash-by\nobreakdash-$\mathbb{Z}_{\mu}$\nobreakdash-group if~$\operatorname{Im}\Delta = \{1\}$, or~an~$((F \times \mathbb{Z})$\nobreakdash-by\nobreakdash-$\mathbb{Z}_{\mu})$\nobreakdash-by\nobreakdash-$\mathbb{Z}_{2}$\nobreakdash-group if~$\operatorname{Im}\Delta = \{1,-1\}$. Hence, $G$ is~residually a~$\mathcal{C}$\nobreakdash-group by~Proposition~\ref{proposition204}.

3.\hspace{1ex}Since $\operatorname{Im}\Delta \not\subseteq \{1,-1\}$, Propositions~\ref{proposition401} and~\ref{proposition402} imply that $G$ contains a~subgroup isomorphic to~$\mathrm{BS}(m,n)$, where $1 < m < |n|$. This subgroup is~not~residually a~$\mathcal{C}$\nobreakdash-group by~Theorem~\ref{theorem01}. Therefore, $G$ is~also not~residually a~$\mathcal{C}$\nobreakdash-group.
\end{proof}

\begin{proof}[\textup{\textbf{Proof of~Theorem~\ref{theorem04}.}}]
1.\hspace{1ex}Since $\mathcal{C}$ contains at~least one non-pe\-ri\-od\-ic group and~is closed under~taking subgroups and~extensions, it~contains an~infinite cyclic group and~both of~its extensions by~means of~an~infinite cyclic group. Therefore, every elementary GBS-group is~a~tor\-sion-free $\mathcal{C}$\nobreakdash-group.

2.\hspace{1ex}Choose some maximal subtree~$T$ in~$\Gamma$ and~transform the~labeled graph~$\mathcal{L}(\Gamma)$ defining~$G$ to~a~$T$\nobreakdash-posi\-tive form. Since the~fundamental groups of~the~original and~modified labeled graphs are isomorphic, the~subring~$Q$ remains unchanged under~the~indicated transformation and~therefore $Q^{+} \in \mathcal{C}$.

Let $A$, $X$, and~$\sigma\colon G \to X$ be~the~groups and~the~homomorphism from~the~statement of~Proposition~\ref{proposition501}. By~the~definition, $\sigma$ acts injectively on~all the~vertex groups, and~Proposition~\ref{proposition302} says that $\ker\sigma$ is~a~free group.

As~noted above, an~infinite cyclic group belongs to~$\mathcal{C}$. By~the~definition of~root class, $\mathcal{C}$ also contains the~Cartesian product $P = \prod_{z \in \mathbb{Z}} C_{z}$, where $C_{z}$ is~an~infinite cyclic group for~each $z \in \mathbb{Z}$. Since $A$ is~isomorphic to~a~subgroup of~$P$ and~$\mathcal{C}$ is~closed under~taking subgroups and~extensions, then $A$, $X$, and~$\operatorname{Im}\sigma$ belong to~$\mathcal{C}$. Therefore, $G$ is~residually a~$\mathcal{C}$\nobreakdash-group by~Proposition~\ref{proposition204}.

It remains to~note that $Q^{+}$ is~a~homomorphic image of~a~free abelian group of~countable rank, which is~a~subgroup of~$P$. Therefore, if~$\mathcal{C}$ is~closed under~taking quotient groups, then $Q^{+} \in \mathcal{C}$. If~$\operatorname{Im}\Delta \subseteq \{1,-1\}$, then $Q^{+}$ is~an~infinite cyclic group, that belongs to~$\mathcal{C}$ as~noted above.
\end{proof}

\begin{proof}[\textup{\textbf{Proof of~Corollaries~\ref{corollary01}\nobreakdash--\ref{corollary03}.}}]
Let $\rho$ be~a~non-empty set of~primes. Using the~definition of~root class, it~is~easy to~verify that the~classes of~finite $\rho$\nobreakdash-groups, finite solvable $\rho$\nobreakdash-groups, periodic $\rho$\nobreakdash-groups of~finite exponent, periodic solvable $\rho$\nobreakdash-groups of~finite exponent, and~all solvable groups are root. Therefore, the~implications $1 \Rightarrow 5$, $3 \Rightarrow 5$, $5 \Rightarrow 2$, and~$5 \Rightarrow 4$ in~Corollary~\ref{corollary02} follow from~Theorem~\ref{theorem03}, the~implications $1 \Rightarrow 3$ and~$3 \Rightarrow 2$ in~Corollary~\ref{corollary01} do from~Theorems~\ref{theorem01} and~\ref{theorem03}, and~Corollary~\ref{corollary03} follows from~Theorem~\ref{theorem04}. The~implications $4 \Rightarrow 3$, $2 \Rightarrow 1$ in~Corollary~\ref{corollary02} and~$2 \Rightarrow 1$ in~Corollary~\ref{corollary01} are obvious.
\end{proof}

\section{An~algorithm for~verifying the~condition of~Theorem~\ref{theorem05}}\label{section06}

Let $E^{*}$ be~the~set of~paths in~$\Gamma$. We define a~function $\xi\colon E^{*} \to \{1,-1\}$ as~follows. If~$e \in E$, then $\xi(e) = \operatorname{sign}\lambda(+e)\lambda(-e)$. If~$s = (e_{1}, e_{2}, \ldots, e_{n})$ is~a~path in~$\Gamma$, then $\xi(s) = \prod_{i = 1}^{n} \xi(e_{i})$. In~particular, if~the~length of~$s$ is~equal to~$0$, then $\xi(s) = 1$.

The~algorithm given below assigns labels to~the~vertices of~$\Gamma$. The~label corresponding to~a~vertex~$v$ is~denoted by~$\zeta(v)$ and~can be~equal to~$\pm 1$. Initially, all the~vertices are unlabeled.

\enlargethispage{1pt}

\begin{algorithm1}
1.\hspace{1ex}If the~graph has no~labeled vertices, take an~arbitrary vertex~$v$ of~$\Gamma$. Otherwise, choose some vertex~$v$, which has no~label and~is~adjacent to~one of~the~labeled vertices.

2.\hspace{1ex}If there is~a~loop~$e$ at~$v$ such that $\xi(e) = -1$, then the~algorithm terminates without labeling~$v$.

3.\hspace{1ex}Let $E_{v}$ be~the~set of~edges of~$\Gamma$, each of~which connects~$v$ with~some of~the~already labeled vertices, and~let, for~any $e \in E_{v}$, $\varepsilon_{e}$ denote the~number that is~equal to~$\pm 1$ and~satisfies the~relation $e(\varepsilon_{e}) = v$.

3.1.\hspace{1ex}If there exist $e_{1}$,~$e_{2} \in E_{v}$ such that $\xi(e_{1})\zeta(e_{1}(-\varepsilon_{e_1})) \ne \xi(e_{2})\zeta(e_{2}(-\varepsilon_{e_2}))$, then the~algorithm terminates without labeling~$v$.

3.2.\hspace{1ex}Otherwise, we define $\zeta(v)$ as~follows: $\zeta(v) = 1$ if~$E_{v} = \varnothing$, and~$\zeta(v) = \xi(e)\zeta(e(-\varepsilon_{e}))$ if~$E_{v} \ne \varnothing$ and~$e$ is~some edge from~$E_{v}$ (the~independence of~$\zeta(v)$ from~the~choice of~$e$ is~ensured by~Step~3.1).

4.\hspace{1ex}If all the~vertices of~$\Gamma$ are labeled, then the~algorithm terminates; otherwise, it~returns to~Step~1.
\end{algorithm1}

\begin{proposition}\label{proposition601}
\textup{1.}\hspace{1ex}If $\xi(s) = 1$ for~any closed path $s$ in~$\Gamma$, then the~above algorithm terminates by~labeling all the~vertices of~$\Gamma$ for~any sequence of~vertex selection at~Step~\textup{1}.

\textup{2.}\hspace{1ex}If, for~some sequence of~vertex selection at~Step~\textup{1}, the~above algorithm terminates by~labeling all the~vertices of~$\Gamma$, then $\xi(s) = 1$ for~any closed path~$s$ in~$\Gamma$.
\end{proposition}

\begin{proof}
First of~all, we note that if~$v$ is~some vertex of~$\Gamma$ labeled by~the~algorithm and~$u$ is~the~vertex that was labeled first, then $\Gamma$ contains a~path~$s$ from~$u$ to~$v$ consisting of~vertices labeled no later than~$v$, and~besides $\zeta(v) = \xi(s)$. This is~not~difficult to~show using induction on~the~number of~steps of~the~algorithm.

1.\hspace{1ex}Let us fix some sequence of~vertex selection at~Step~1 and~consider an~arbitrary vertex~$v$ from~this sequence. If~there is~a~loop~$e$ at~$v$, then it~is~a~closed path and~therefore $\xi(e) = 1$. Let $E_{v} \ne \varnothing$, $e_{1}, e_{2} \in E_{v}$ be~arbitrary edges, $v_{1} = e_{1}(-\varepsilon_{e_1})$, $v_{2} = e_{2}(-\varepsilon_{e_2})$, and~$u$ be~the~first vertex labeled by~the~algorithm. Then there exist paths $s_{1}$,~$s_{2}$ connecting $v_{1}$,~$v_{2}$ with~$u$ and~such that $\zeta(v_{1}) = \xi(s_{1})$, $\zeta(v_{2}) = \xi(s_{1})$. Let $s$ be~the~path composed of~the~paths $s_{1}$,~$s_{2}$ and~the~edges $e_{1}$,~$e_{2}$. This path is~closed, so $1 = \xi(s) = \xi(s_{1})\xi(s_{2})\xi(e_{1})\xi(e_{2})$ and~$\xi(e_{1})\zeta(v_{1}) = \xi(e_{1})\xi(s_{1}) = \xi(e_{2})\xi(s_{2}) = \xi(e_{2})\zeta(v_{2})$. Thus, the~algorithm terminates neither at~Step~2 nor~at~Step~3.1, and~the~vertex~$v$ is~among the~labeled ones. Since it~is~chosen arbitrarily, this means that the~algorithm labels all the~vertices of~$\Gamma$.

2.\hspace{1ex}Since the~algorithm labels all the~vertices of~the~graph without terminating at~Step~2, then $\xi(e) = 1$ for~every loop $e \in E$ and~further we can consider only closed paths that do not~contain loops. We argue by~induction on~the~number~$n$ of~iterations (Steps~1\nobreakdash--4) required for~the~algorithm to~label all the~vertices of~a~path of~the~indicated~form.

Let $s$ be~a~closed path without loops, the~vertices of~which are labeled in~$n$ iterations. At~each iteration of~the~algorithm, no more than one vertex is~labeled. Therefore, if~$n = 1$, then the~length of~$s$ is~equal to~$0$ and~the~equality $\xi(s) = 1$ is~obvious. Further, we assume that $s$ is~of~non-zero length (so $n > 1$) and~$\xi$ has the~required value for~every closed path whose vertices are labeled in~at most $n - 1$ iterations.

Let $v$ be~the~last labeled vertex of~$s$. If~necessary, we split $s$ into~the~closed parts, each of~which begins and~ends at~$v$, and~assume that $s$ passes through~$v$ only once. Then the~fragment $(v_{1}, e_{1}, v, e_{2}, v_{2})$ of~$s$ is~uniquely defined, where $e_{1}$,~$e_{2}$ are edges (which may coincide) and~$v_{1}$,~$v_{2}$ are vertices (which may also coincide). Since there are no loops in~$s$, the~relations $v_{1} \ne v \ne v_{2}$ and~$e_{1}, e_{2} \in E_{v}$ hold.

Let $u$ be~the~first vertex labeled by~the~algorithm. Then there exist paths $s_{1}$,~$s_{2}$ connecting $v_{1}$,~$v_{2}$ with~$u$, consisting of~the~vertices labeled no later than $v_{1}$,~$v_{2}$ respectively, and~such that $\zeta(v_{1}) = \xi(s_{1})$, $\zeta(v_{2}) = \xi(s_{2})$. Let $s_{0}$ denote the~path obtained from~$s$ by~removing~$v$ and~$e_{1}$,~$e_{2}$. Then the~union~$s_{3}$ of~the~paths $s_{0}$, $s_{1}$, $s_{2}$ is~a~closed path, all the~vertices of~which are labeled in~at most $n - 1$ iterations, and,~by~the~inductive hypothesis, $1 = \xi(s_{3}) = \xi(s_{0})\xi(s_{1})\xi(s_{2})$. Since $v$ is~labeled by~the~algorithm, then the~condition of~Step~3.1 cannot be~satisfied and~so $\xi(e_{1})\zeta(v_{1}) = \xi(e_{2})\zeta(v_{2})$. Hence, $\xi(e_{1})\xi(e_{2}) = \zeta(v_{1})\zeta(v_{2}) = \xi(s_{1})\xi(s_{2})$, and~therefore $\xi(s) = \xi(s_{0})\xi(e_{1})\xi(e_{2}) = \xi(s_{0})\xi(s_{1})\xi(s_{2}) = 1$, as~required.
\end{proof}

\begin{proposition}\label{proposition602}
{\parfillskip=0pt
Let $G$ be~not~solvable, $\mathcal{L}(\Gamma)$ be~reduced, $\operatorname{Im}\Delta = \{1,-1\}$, and~all the~labels $\lambda(\varepsilon e)$ \textup{(}$e \in E$, $\varepsilon = \pm 1$\textup{)} are $p$\nobreakdash-num\-bers for~some prime number $p \ne 2$. Let~also
$$
E^{\prime} = \{e \in E \mid H_{+e} \ne C(G) \ne H_{-e}\}.
$$\par}

\textup{1.}\hspace{1ex}If every elliptic element of~$G$ that~is conjugate to~its inverse belongs to~$C(G)$, then $\xi(s) = 1$ for~every closed path~$s$ in~$\Gamma$ all of~whose edges are contained in~$E^{\prime}$.

\textup{2.}\hspace{1ex}If $\xi(s) = 1$ for~every closed path~$s$ in~$\Gamma$ all of~whose edges are contained in~$E^{\prime}$, then every elliptic element of~$G$ that~is conjugate to~its inverse belongs to~$C(G)$ and~the~quotient group~$G/C(G)$ is~residually a~finite $p$\nobreakdash-group.
\end{proposition}

\begin{proof}
We fix some maximal subtree~$T$ of~$\Gamma$ and~begin with~a~few remarks concerning both Statement~1 and~Statement~2.

By Proposition~\ref{proposition404},
$$
1 \ne C(G) = \bigcap_{\substack{e \in E,\\ \varepsilon = \pm 1}} H_{\varepsilon e} \leqslant \bigcap_{v \in V} G_{v}
$$
and~the~least common multiple~$\mu$ of~$\mu(v) = [G_{v}:C(G)]$ ($v \in V$) divides the~product $\prod_{e \in E,\,\varepsilon = \pm 1} \lambda(\varepsilon e)$. Therefore, $\mu$ and~all the~indices~$\mu(v)$ ($v \in V$) are $p$\nobreakdash-num\-bers.

If $\Gamma$ contains one vertex~$v$ and~$\mu = \mu(v) = 1$, then $E^{\prime} = \varnothing$, every elliptic element of~$G$ belongs to~$G_{v} = C(G)$, and~$G/C(G)$ is~a~free group, that is~residually a~finite $p$\nobreakdash-group by~Proposition~\ref{proposition204}. Therefore, both Statement~1 and~Statement~2 are true. If~$\Gamma$ has more than one vertex, then each of~its vertices is~incident to~some edge that is~not~a~loop. Since $\mathcal{L}(\Gamma)$ is~reduced, it~follows that every vertex group contains some proper edge subgroup and~so $C(G) \ne G_{v}$ for~all $v \in V$. Thus, further, we can assume that all $\mu(v)$ ($v \in V$) are different from~$1$ and~therefore are divisible by~$p$.

Let $e \in E$ be~an~arbitrary edge. Then $g_{e(1)}^{\lambda(+e)} \sim_{G} g_{e(-1)}^{\lambda(-e)}$ and~$H_{\vphantom{(}+e}^{\vphantom{(}} \sim_{G} H_{\vphantom{(}-e}^{\vphantom{(}}$. Since $C(G)$ is~normal in~$G$, it~follows that $[H_{+e}:C(G)] = [H_{-e}:C(G)] = k_{e}$ for~some $p$\nobreakdash-num\-ber $k_{e} \geqslant 1$ and~$|\lambda(+e)|k_{e} = \mu(e(1))$, $|\lambda(-e)|k_{e} = \mu(e(-1))$.

Let us now turn directly to~the~proof of~Statements~1 and~2.

1.\hspace{1ex}We put $g_{v}^{\prime} = g_{\vphantom{i}v}^{\mu(v)/p}$ ($v \in V$) and~show that, for~any edge $e \in E^{\prime}$, $g_{e(1)}^{\prime}$ and~$(g_{e(-1)}^{\prime})^{\xi(e)}$ are conjugate in~$G$.

Indeed, let $e \in E^{\prime}$ be~an~arbitrary edge. Then $k_{e} \ne 1$ and~so $p \mid k_{e}$. The~relation $g_{e(1)}^{\lambda(+e)} \sim_{G} g_{e(-1)}^{\lambda(-e)}$ implies that $g_{e(1)}^{|\lambda(+e)|} \sim_{G} g_{e(-1)}^{\xi(e)|\lambda(-e)|}$. Hence,
$$
g_{e(1)}^{\vphantom{y}\prime} = g_{e(1)}^{|\lambda(+e)|(k_e/p)} \sim_{G} g_{e(-1)}^{\xi(e)|\lambda(-e)|(k_e/p)} = \big(g_{e(-1)}^{\vphantom{y}\prime}\big)^{\xi(e)}.
$$

Thus, if~$s$ is~a~closed path in~$\Gamma$, all the~edges of~which are contained in~$E^{\prime}$, and~$\xi(s) = -1$, then, for~every vertex~$v$ of~this path, $g_{v}^{\prime} \sim_{G} (g_{v}^{\prime})^{\xi(s)} = (g_{v}^{\prime})^{-1}$. Since $g_{v}^{\prime} \notin C(G)$, Statement~1 is~proved.

2.\hspace{1ex}Let $E_{T}$ denote the~edge set of~$T$. To~prove the~residual $p$\nobreakdash-fi\-nite\-ness of~the~quotient group~$G/C(G)$, we define a~mapping~$\sigma_{0}$ of~the~generators~$g_{v}$ ($v \in V$) and~$t_{e}$ ($e \in E \setminus E_{T}$) of~$G$ to~$\mathbb{Z}_{\mu}$ as~follows.

Let $\Gamma^{\prime} = (V, E^{\prime})$ be~the~graph obtained from~$\Gamma$ by~removing all the~edges not~included in~$E^{\prime}$, and~let $\Gamma^{\prime}_{i} = (V_{i}^{\vphantom{\prime}}, E_{i}^{\prime})$ be~some connected component of~$\Gamma^{\prime}$. Choose an~arbitrary vertex $v \in V_{i}$ and~put $g_{v}\sigma_{0} = \mu/\mu(v)$. If~$w \in V_{i}$ is~an~arbitrary vertex and~$s$ is~a~path in~$\Gamma^{\prime}_{i}$ connecting $v$ and~$w$, we put $g_{w}\sigma_{0} = \xi(s)\mu/\mu(w)$. It~follows from~the~condition of~Statement~2 that, for~any two paths $s_{1}$, $s_{2}$ connecting $v$ and~$w$ in~$\Gamma^{\prime}_{i}$, the~equality $\xi(s_{1}) = \xi(s_{2})$ holds. Therefore, the~above definition is~correct. In~a~similar way, we define the~action of~$\sigma_{0}$ on~the~generators of~the~vertex groups contained in~all other connected components of~$\Gamma^{\prime}$. Let us also put $t_{e}\sigma_{0} = 0$ for~all $e \in E \setminus E_{T}$.

We extend $\sigma_{0}$ to~a~mapping of~words~$\sigma$ and~show that the~latter takes all the~defining relations of~$G$ into~the~equalities valid in~$\mathbb{Z}_{\mu}$.

{\parfillskip=0pt
Let $e \in E$ be~an~arbitrary edge. As~shown earlier, $|\lambda(+e)|k_{e} = \mu(e(1))$ and~$|\lambda(-e)|k_{e} = \mu(e(-1))$, where $k_{e} = [H_{+e}:C(G)] = [H_{-e}:C(G)]$. If~$e \in E^{\prime}$, $v$ is~the~fixed vertex chosen above from~the~connected component of~$\Gamma^{\prime}$, to~which $e$ belongs, and~$s_{1}$, $s_{-1}$ are some paths in~$\Gamma^{\prime}$ connecting $v$ with~$e(1)$, $e(-1)$ respectively, then $\xi(s_{1}) = \xi(s_{-1})\xi(e)$. It~follows that
$$
\xi(s_{1}) \cdot \operatorname{sign}\lambda(+e) = \xi(s_{-1}) \cdot \operatorname{sign}\lambda(-e)
$$
and\hfill\mbox{}
\begin{align*}
g_{e(1)}^{\lambda(+e)}\sigma &= \xi(s_{1})\lambda(+e)\mu/\mu(e(1))\\
&= \xi(s_{1}) \cdot \operatorname{sign}\lambda(+e) \cdot |\lambda(+e)|\mu/\mu(e(1))\\
&= \xi(s_{-1}) \cdot \operatorname{sign}\lambda(-e) \cdot |\lambda(-e)|\mu/\mu(e(-1))\\
&= \xi(s_{-1})\lambda(-e)\mu/\mu(e(-1))\\
&= g_{e(-1)}^{\lambda(-e)}\sigma.
\end{align*}\par}

If~$e \notin E^{\prime}$, then $k_{e} = 1$ and~therefore
$$
g_{e(1)}^{\lambda(+e)}\sigma = \varepsilon|\lambda(+e)|\mu/\mu(e(1)) = \varepsilon\mu \equiv \delta\mu = \delta|\lambda(-e)|\mu/\mu(e(-1)) = g_{e(-1)}^{\lambda(-e)}\sigma \pmod \mu
$$
for~some $\varepsilon, \delta = \pm 1$.

Thus, $\sigma$ defines a~homomorphism of~$G$ into~the~finite $p$\nobreakdash-group~$\mathbb{Z}_{\mu}$. It~follows from~the~definition of~$\sigma$ that, for~each $v \in V$, the~order of~$g_{v}\sigma$ is~equal to~$\mu(v)$ and~therefore $\ker\sigma \cap G_{v} = C(G)$. Hence, according to~Proposition~\ref{proposition404}, $\ker\sigma$ is~an~extension of~$C(G)$ by~a~free group. This implies that the~quotient group~$G/C(G)$ is~an~extension of~the~indicated free group by~a~finite $p$\nobreakdash-group. Such an~extension is~residually a~finite $p$\nobreakdash-group by~Proposition~\ref{proposition204}.

Suppose now that $x$ and~$y$ are elements of~$G$ such that $x^{-1}yx = y^{-1}$. Then
$$
\big(xC(G)\big)^{-1}\big(yC(G)\big)\big(xC(G)\big) = \big(yC(G)\big)^{-1}\kern-4pt,
$$
and~the~residual $p$\nobreakdash-fi\-nite\-ness of~$G/C(G)$ proved above together with~Proposition~\ref{proposition206} and~the~relation $p \ne 2$ imply that $yC(G) = 1$, i.\;e.~$y \in C(G)$. Thus, Statement~2 is~completely proved.
\end{proof}

\begin{algorithm2}
Let $G$ be~not~solvable, $\mathcal{L}(\Gamma)$ be~reduced, and~$\operatorname{Im}\Delta = \{1,-1\}$. Then $C(G) \leqslant \bigcap_{v \in V} G_{v}$ and~there is~an~algorithm calculating the~numbers $\mu(v) = [G_{v}:C(G)]$ ($v \in V$)~\cite[\S~5]{DelgadoRobinsonTimm2017}. This allows us to~find the~graph~$\Gamma^{\prime}$ which is~obtained from~$\Gamma$ by~removing all the~edges not~included in~the~set $E^{\prime} = \{e \in E \mid H_{+e} \ne C(G) \ne H_{-e}\}$. By~Propositions~\ref{proposition601} and~\ref{proposition602}, to~complete the~verification of~the~condition of~Statement~2\nobreakdash-\textit{c}, it~remains to~apply Algorithm given above to~each connected component of~$\Gamma^{\prime}$.
\end{algorithm2}

\section{Proof of~Theorems~\ref{theorem05} and~\ref{theorem06}}\label{section07}

\begin{proposition}\label{proposition701}
Let $G$ be~non-solv\-a\-ble and~$\mathcal{L}(\Gamma)$ be~reduced. If~$G$ is~residually nilpotent, then all the~labels $\lambda(\varepsilon e)$ \textup{(}$e \in E$, $\varepsilon = \pm 1$\textup{)} are $p$\nobreakdash-num\-bers for~some prime number~$p$.
\end{proposition}

\begin{proof}
Since $\mathcal{L}(\Gamma)$ is~reduced, then $|\lambda(+e)| \ne 1 \ne |\lambda(-e)|$ for~each edge~$e$ that is~not~a~loop. Let us show that these relations can be~assumed to~hold for~all loops of~$\mathcal{L}(\Gamma)$.

By~Proposition~\ref{proposition205}, $G$ is~residually finite, and,~by~Theorem~\ref{theorem03}, $\operatorname{Im}\Delta \subseteq \{1,-1\}$. It~follows that $\mathcal{L}(\Gamma)$ cannot contain a~loop~$e$ such that $|\lambda(\varepsilon e)| = 1 \ne |\lambda(-\varepsilon e)|$ for~some $\varepsilon = \pm 1$. Let $\Gamma^{\prime}$ be~the~subgraph of~$\Gamma$ which is~obtained from~the~latter by~removing each loop~$e$ such that (in~$\mathcal{L}(\Gamma))$ $|\lambda(+e)| = 1 = |\lambda(-e)|$, and~let $G^{\prime} = \pi_{1}(\mathcal{L}(\Gamma^{\prime}))$. Then, by~Proposition~\ref{proposition301}, $G^{\prime}$ is~isomorphic to~a~subgroup of~$G$ and~so is~residually nilpotent. Since\linebreak $|\lambda(+e)| \ne 1 \ne |\lambda(-e)|$ for~every edge~$e$ of~$\Gamma^{\prime}$ and~$1$ is~a~power of~any number~$p$, then we can consider $\Gamma^{\prime}$, $\mathcal{L}(\Gamma^{\prime})$, and~$G^{\prime}$ instead of~$\Gamma$, $\mathcal{L}(\Gamma)$, and~$G$ respectively.

So, we assume that $|\lambda(\varepsilon e)| \ne 1$ for~all $e \in E$, $\varepsilon = \pm 1$. If, for~any edge $e \in E$, at~least one of~the~numbers $|\lambda(+e)|$, $|\lambda(-e)|$ is~greater than~$2$, then the~required statement follows from~Proposition~\ref{proposition303}. Therefore, it remains to show that if $|\lambda(+e)| = 2 = |\lambda(-e)|$ for~some edge $e \in E$, then all the~labels $\lambda(\varepsilon f)$ $(f \in E$, $\varepsilon = \pm 1)$ are $2$\nobreakdash-num\-bers.

On~the~contrary, let $f \in E$ be~an~edge such that at~least one of~the~numbers $\lambda(+f)$, $\lambda(-f)$ is~divisible by~a~prime number $p \ne 2$. Let us show that $\Gamma$ and~$\mathcal{L}(\Gamma)$ can, if~necessary, be~modified so that a)~$\Gamma$ contains a~simple chain, the~first and~last edges of~which are $e$ and~$f$; b)~the~fundamental groups of~the~original and~modified labeled graphs are isomorphic.

Indeed, if~there~is no~chain of~the~indicated type in~$\Gamma$, then at~least one of~the~following statements holds: 1)~$e$ and~$f$ are not~loops and~connect identical pairs of~vertices, i.\;e.~$f(1) = e(\varepsilon)$ and~$f(-1) = e(-\varepsilon)$ for~some $\varepsilon = \pm 1$; 2)~$e$~is~a~loop; 3)~$f$~is~a~loop. In~the~first case, we add to~$\Gamma$ a~new vertex~$v_{f}$ and~an~edge connecting this vertex with~$f(-1)$; replace~$f$ with~an~edge connecting $f(1)$ and~$v_{f}$; in~$\mathcal{L}(\Gamma)$, assign the~labels~$(1,1)$ to~the~first of~the~added edges, $\lambda(+f)$ (at $f(1)$) and~$\lambda(-f)$ (at $v_{f}$) to~the~second. We perform exactly the~same transformations in~the~third case and~modify $\Gamma$ and~$\mathcal{L}(\Gamma)$ in~a~similar way in~the~second. In~all cases, the~original labeled graph is~obtained from~the~modified one by~an~elementary collapse; therefore, their fundamental groups are isomorphic.

Let $\Omega$ be~a~simple chain in~$\Gamma$ that begins with~$e$ and~ends with~$f$. By~Proposition~\ref{proposition301}, $\pi_{1}(\mathcal{L}(\Omega))$ is~embedded in~$\pi_{1}(\mathcal{L}(\Gamma))$ and~so is~residually nilpotent. For~definiteness, let $e(1)$ and~$f(-1)$ be~the~ends of~the~chain, and~let $\varepsilon = \pm 1$ be~a~number such that $p \mid \lambda(\varepsilon f)$. Consider the~elements
\begin{gather*}
x_{\vphantom{(}1}^{\vphantom{(}} = 
\big[g_{e(1)}^{\vphantom{\lambda(f)}},\,
g_{e(-1)}^{\vphantom{\lambda(f)}}\big],\quad
x_{\vphantom{(}2}^{\vphantom{(}} = 
\big[g_{f(-\varepsilon)}^{\vphantom{\lambda(f)}},\, 
g_{f(\varepsilon)}^{\lambda(\varepsilon f)/p}\big],\\
x = \big[x_{\vphantom{(}1}^{\vphantom{(}},\,
x_{\vphantom{(}2}^{\vphantom{(}}\big] = 
g_{e(-1)}^{-1}
g_{e(1)}^{-1}
g_{\vphantom{(}e(-1)}^{\vphantom{(}}
g_{\vphantom{(}e(1)}^{\phantom{(}}
x_{\vphantom{(}2}^{-1}
g_{e(1)}^{-1}
g_{e(-1)}^{-1}
g_{\vphantom{(}e(1)}^{\phantom{(}}
g_{\vphantom{(}e(-1)}^{\phantom{(}}
x_{\vphantom{(}2}^{\vphantom{(}}.
\end{gather*}

Let $\Omega_{1}$ be~the~chain obtained from~$\Omega$ by~removing $e(1)$ and~$e$, $\Omega_{2}$ be~the~chain obtained from~$\Omega$ by~removing $f(-1)$ and~$f$, $F_{1} = \pi_{1}(\mathcal{L}(\Omega_{1}))$, and~$F_{2} = \pi_{1}(\mathcal{L}(\Omega_{2}))$. Then $\pi_{1}(\mathcal{L}(\Omega))$ is~the~free product~$P_{1}$ of~the~groups $G_{e(1)}$, $F_{1}$ with~the~subgroups $H_{+e}$, $H_{-e}$ amalgamated and,~at~the~same time, the~free product~$P_{2}$ of~the~groups $F_{2}$, $G_{f(-1)}$ with~the~subgroups $H_{+f}$, $H_{-f}$ amalgamated. Since $|\lambda(-\varepsilon f)| \ne 1$ and~$|\lambda(\varepsilon f)/p| < |\lambda(\varepsilon f)|$, then $x_{2}$ has a~reduced form of~length~$4$ in~$P_{2}$ and~hence does not~belong to~the~free factor~$F_{2}$ and~its subgroup~$H_{-e}$. It~follows from~this and~the~equalities $|\lambda(+e)| = 2 = |\lambda(-e)|$ that $x$ has a~reduced form of~length at~least~$8$ in~$P_{1}$ and~therefore is~different from~$1$.

Let $q$ be~an~arbitrary prime number and~$\psi$ be~a~homomorphism of~$\pi_{1}(\mathcal{L}(\Omega))$ onto~a~finite $q$\nobreakdash-group. If~$q \ne 2$ and~$r$ is~the~order of~$g_{e(1)}\psi$, then $(r,2) = 1$. This equality and~the~inclusions $g_{e(1)}^{2}\psi \in H_{\vphantom{(}+e}^{\vphantom{1}}\psi$, $g_{e(1)}^{r}\psi \in H_{\vphantom{(}+e}^{\vphantom{1}}\psi$ imply that $g_{e(1)}\psi \in H_{+e}\psi = H_{-e}\psi$ and~$x_{1}\psi = 1$. Similarly, if~$q \ne p$ and~$s$ is~the~order of~$g_{f(\varepsilon)}^{\lambda(\varepsilon f)/p}\psi$, then it~follows from~the~inclusions
$$
\big(g_{f(\varepsilon)}^{\lambda(\varepsilon f)/p}\big)^{p}\psi \in H_{\vphantom{(}\varepsilon f}^{\vphantom{(}}\psi,\quad
\big(g_{f(\varepsilon)}^{\lambda(\varepsilon f)/p}\big)^{s}\psi \in H_{\vphantom{(}\varepsilon f}^{\vphantom{(}}\psi
$$
that $x_{2}\psi = 1$. Thus, for~each homomorphism of~$\pi_{1}(\mathcal{L}(\Omega))$ onto~a~finite group of~prime power order, the~image of~$x$ turns~out to~be equal to~$1$. This contradicts the~residual nilpotence of~$\pi_{1}(\mathcal{L}(\Omega))$ by~Proposition~\ref{proposition205}.
\end{proof}

\begin{proof}[\textup{\textbf{Proof of~Theorem~\ref{theorem05}.}}]
1.\hspace{1ex}If $G$ is~residually nilpotent, then, by~Proposition~\ref{proposition701}, all the~labels~$\lambda(\varepsilon e)$ ($e \in E$, $\varepsilon = \pm 1$) are $p$\nobreakdash-num\-bers for~some prime number~$p$. This fact, the~equality $\operatorname{Im}\Delta = \{1\}$, and~Theorem~\ref{theorem03} imply the~residual $p$\nobreakdash-fi\-nite\-ness of~$G$. Since every finite $p$\nobreakdash-group is~nilpotent, the~inverse statement is~obvious.

2.\hspace{1ex}The~implication $b \Rightarrow a$ is~obvious.

$a \Rightarrow c$.\hspace{1ex}By~Proposition~\ref{proposition701}, all the~labels~$\lambda(\varepsilon e)$ ($e \in E$, $\varepsilon = \pm 1$) are $p$\nobreakdash-num\-bers for~some prime number~$p$. Suppose that $p \ne 2$ and~there exists an~elliptic element~$a$, which is~conjugate to~its inverse but~does not~belong to~the~cyclic radical of~$G$.

Let $T$ be~some fixed maximal subtree of~$\Gamma$ and~$E_T$ be~the~edge set of~$T$. Replacing, if~necessary, $a$ by~its conjugate, we can assume that $a \in G_{v}$ for~some $v \in V$. Since $C(G)$ is~normal in~$G$, then $a$ still does not~belong to~$C(G)$ after the~replacement. We put
\begin{align*}
E_{1} &= \big\{e \in E\ \big|\ |\lambda(+e)| = 1 = |\lambda(-e)|\big\},\\[-2pt]
E_{2} &= \big\{e \in E\ \big|\ |\lambda(+e)| \ne 1 \ne |\lambda(-e)|\big\}
\end{align*}
and~show that there exists an~edge $e \in E_{2}$ such that $a \notin H_{\varepsilon e}$ for~some $\varepsilon = \pm 1$.

Indeed, $C(G) = \bigcap_{e \in E,\,\varepsilon = \pm 1} H_{\varepsilon e}$ by~Proposition~\ref{proposition404}. Since $\mathcal{L}(\Gamma)$ is~reduced, every edge $e \in E$ that is~not~a~loop belongs to~$E_{2}$. In~particular, $E_{T} \subseteq E_{2}$. If~$e$ is~a~loop,\linebreak then $\Delta(t_{e}) = \lambda(-e)/\lambda(+e)$ and~it~follows from~the~equality $\operatorname{Im}\Delta = \{1,-1\}$ that either $e \in E_{1}$, or~$e \in E_{2}$. If~$E = E_{1}$, then $\Gamma$ has only one vertex, $C(G)$ coincides with~the~corresponding vertex group and~therefore contains all the~elliptic elements of~$G$, what contradicts the~relation $a \notin C(G)$. Hence, either $\mathcal{L}(\Gamma)$ has one vertex and~at~least one loop $e \in E_{2}$, or~it~contains at~least two vertices and~then each vertex is~incident to~some edge $e \in E_{T} \subseteq E_{2}$. In~both cases, $C(G) = \bigcap_{e \in E_{2},\,\varepsilon = \pm 1} H_{\varepsilon e}$, and~this implies the~existence of~the~sought edge~$e$.

We now consider two cases.

\textit{Case~1}.\hspace{1ex}$e \in E_{T}$.

It is~easy to~see that there~is a~simple chain~$\Omega$ in~$T$ containing $e$ and~such that one of~its ends coincides with~$v$, while the~other does with~$e(\delta)$ for~some $\delta = \pm 1$. By~Proposition~\ref{proposition301}, $\pi_{1}(\mathcal{L}(\Omega))$ is~embedded in~$G$ by~means of~the~identity mapping of~the~generators.

Let $\Omega^{\prime}$ be~the~chain obtained from~$\Omega$ by~removing~$e$ and~$e(\delta)$. Then $\pi_{1}(\mathcal{L}(\Omega))$ is~the~free product of~the~groups $\pi_{1}(\mathcal{L}(\Omega^{\prime}))$ and~$G_{e(\delta)}$ with~the~subgroups $H_{-\delta e}$ and~$H_{\delta e}$ amalgamated. Consider the~elements
$$
x_{\vphantom{(}1}^{\vphantom{y}} = \big[g_{e(\delta)}^{\vphantom{-1}}, g_{e(-\delta)}^{\vphantom{-1}}\big],\quad
x_{\vphantom{(}2}^{\vphantom{y}} = 
\big[x_{\vphantom{(}1}^{\vphantom{y}},a\big] = 
g_{e(-\delta)}^{-1}g_{e(\delta)}^{-1}g_{e(-\delta)}^{\vphantom{-1}}g_{e(\delta)}^{\vphantom{-1}}a^{-1}g_{e(\delta)}^{-1}g_{e(-\delta)}^{-1}g_{e(\delta)}^{\vphantom{-1}}g_{e(-\delta)}^{\vphantom{-1}}a.
$$
Since $|\lambda(+e)| \ne 1 \ne |\lambda(-e)|$, $a \notin H_{\varepsilon e}$, and~the~equality $H_{\varepsilon e} = H_{-\varepsilon e}$ holds in~$\pi_{1}(\mathcal{L}(\Omega))$, then $x_{2}$ has a~reduced form of~length at~least~$8$ in~this group and~therefore is~different from~$1$.

Let $q$ be~an~arbitrary prime number and~$\psi$ be~a~homomorphism of~$G$ onto~a~finite $q$\nobreakdash-group. If~$q \ne 2$, then $a\psi = 1$ by~Proposition~\ref{proposition206}. Let $q = 2$ and~$r$ be~the~order of~$g_{e(1)}\psi$. Since $\lambda(+e)$ is~a~$p$\nobreakdash-num\-ber and~$p \ne 2$, then $(r,\lambda(+e)) = 1$. It~follows that $g_{e(1)}\psi \in H_{+e}\psi = H_{-e}\psi$ and~$x_{1}\psi = 1$. Thus, for~any value of~$q$, the~equality $x_{2}\psi = 1$ holds. By~Proposition~\ref{proposition205}, this contradicts the~residual nilpotence of~$G$.

\textit{Case~2}.\hspace{1ex}$e \notin E_{T}$.

Let $x_{\vphantom{(}1}^{\vphantom{\varepsilon}}\kern-1pt{} =\kern-1pt{} \big[t_{\vphantom{(}e}^{\vphantom{y}\varepsilon}
g_{e(-\varepsilon)}^{\vphantom{y}}
t_{\vphantom{(}e}^{\vphantom{y}-\varepsilon}\kern-4pt,\kern1pt{} 
g_{e(\varepsilon)}^{\vphantom{y}}\big]$. It~follows from~the~relations $|\lambda(+e)| \ne 1 \ne |\lambda(-e)|$, $a \notin H_{\varepsilon e}$ that the~element
$$
x_{2\vphantom{(}}^{\vphantom{-1}} =
\big[x_{1\vphantom{(}}^{\vphantom{-1}},a\big] = 
g_{e(\varepsilon)}^{-1}
t_{e\vphantom{(}}^{\vphantom{-1}\varepsilon}
g_{e(-\varepsilon)}^{-1}
t_{e\vphantom{(}}^{\vphantom{1}-\varepsilon}
g_{e(\varepsilon)}^{\vphantom{-1}}
t_{e\vphantom{(}}^{\vphantom{-1}\varepsilon}
g_{e(-\varepsilon)}^{\vphantom{-1}}
t_{e\vphantom{(}}^{\vphantom{1}-\varepsilon}
a^{-1}
t_{e\vphantom{(}}^{\vphantom{-1}\varepsilon}
g_{e(-\varepsilon)}^{-1}
t_{e\vphantom{(}}^{\vphantom{1}-\varepsilon}
g_{e(\varepsilon)}^{-1}
t_{e\vphantom{(}}^{\vphantom{-1}\varepsilon}
g_{e(-\varepsilon)}^{\vphantom{-1}}
t_{e\vphantom{(}}^{\vphantom{1}-\varepsilon}
g_{e(\varepsilon)}^{\vphantom{-1}}
a
$$
has a~reduced form of~length~$8$ in~the~group~$G$ considered as~an~HNN-ex\-ten\-sion with~the~stable letter~$t_{e}$. Therefore, $x_{2} \ne 1$. However, as~above, if~$\psi$ is~a~homomorphism of~$G$ onto~a~finite $2$\nobreakdash-group, then $g_{e(-\varepsilon)}\psi \in H_{-\varepsilon e}\psi$, whence 
$\big(t_{e\vphantom{(}}^{\vphantom{y}\varepsilon}
g_{e(-\varepsilon)}^{\vphantom{-1}}
t_{e\vphantom{(}}^{\vphantom{y}-\varepsilon}\big)
\psi \in H_{\varepsilon e\vphantom{(}}^{\vphantom{-1}}\psi$
and~so $x_{1}\psi = 1$. Thus, in~Case~2, the~image of~$x_{2}$ is~also equal to~$1$ for~any homomorphism of~$G$ onto~a~group of~prime power order, and~we again get a~contradiction with~the~residual nilpotence of~$G$.

$c \Rightarrow b$.\hspace{1ex}Choose some maximal subtree~$T$ in~$\Gamma$ and~transform $\mathcal{L}(\Gamma)$ to~a~$T$\nobreakdash-posi\-tive form. Since this operation consists only in~replacing some of~the~generators~$g_{v}$ ($v \in V$) by~their inverse, then, after it, the~conditions of~Statement~2\nobreakdash-\textit{c} remain valid.

If all the~labels~$\lambda(\varepsilon e)$ ($e \in E$, $\varepsilon = \pm 1$) are $2$\nobreakdash-num\-bers, then, by~Propositions~\ref{proposition404} and~\ref{proposition502}, $G$ is~an~extension of~the~direct product of~two free groups by~a~finite $2$\nobreakdash-group. Such an~extension is~residually a~finite $2$\nobreakdash-group by~Proposition~\ref{proposition204}. So, further, we assume that $p \ne 2$.

Let $g \in G$ be~an~arbitrary non-unit element. We show that there exists a~homomorphism of~$G$ onto~a~finite $p$\nobreakdash-group or~a~finite $2$\nobreakdash-group taking~$g$ to~a~non-unit element.

By~Proposition~\ref{proposition602}, $G/C(G)$ is~residually a~finite $p$\nobreakdash-group. Therefore, if~$g \notin C(G)$, then the~natural homomorphism of~$G$ onto~$G/C(G)$ can be~extended to~the~desired one. Hence, further, we can assume that $g \in C(G)$.

Let $Q$\kern-.5pt{}, $X$\kern-1.5pt{}, and~$\sigma\kern-1pt{}\colon\kern-1pt{} G\kern-1pt{} \to\kern-1.5pt{} X$ be~the~subring, the~group, and~the~homomorphism from~Proposition~\ref{proposition501}. Since $\operatorname{Im}\Delta = \{1,-1\}$, then $Q = \mathbb{Z}$ and~$X$ has the~presentation
$$
\big\langle x, a_{1}^{\vphantom{1}}, a_{-1}^{\vphantom{1}};\ 
[x,a_{1}^{\vphantom{1}}] = 
[a_{1}^{\vphantom{1}},a_{-1}^{\vphantom{1}}] = 1,\ 
a_{-1}^{-1}xa_{-1}^{\vphantom{1}} = x^{-1} \big\rangle
$$
(here, as~above, $x$~denotes the~generator of~the~additive group~$Q^{+}$ of~$Q$ equal to~$1$). By~Proposition~\ref{proposition404}, $C(G) \leqslant \bigcap_{v \in V} G_{v}$. Hence, $g \in G_{v}$ for~each $v \in V$, and,~by~the~definition of~$\sigma$, the~inclusion $g\sigma \in \langle x \rangle \setminus \{1\}$ holds. Therefore, $g\sigma$ is~mapped to~a~non-unit element under~the~homomorphism of~$X$ onto~the~group
$$
\mathrm{BS}(1,-1) = \big\langle x, a_{-1}^{\vphantom{1}};\ 
a_{-1}^{-1}xa_{-1}^{\vphantom{1}} = x^{-1} \big\rangle.
$$
The~latter is~residually a~finite $2$\nobreakdash-group by~Theorem~\ref{theorem01}. Thus, the~constructed homomorphism $G \to \mathrm{BS}(1,-1)$ can be~extended to~the~required~one.

3.\hspace{1ex}Since $\operatorname{Im}\Delta \not\subseteq \{1,-1\}$, then $G$ is~not~residually finite by~Theorem~\ref{theorem03} and~is~not~residually nilpotent by~Proposition~\ref{proposition205}.
\end{proof}

\begin{proof}[\textup{\textbf{Proof of~Theorem~\ref{theorem06}.}}]
$1 \Rightarrow 3$.\hspace{1ex}Since $G$ is~residually a~tor\-sion-free nilpotent group, then, by~Proposition~\ref{proposition205}, it~is~residually a~finite $p$\nobreakdash-group for~any prime number~$p$. Therefore, by~Theorem~\ref{theorem01}, $G$ cannot be~isomorphic to~$\mathrm{BS}(1,n)$, where $n \ne 1$. Obviously, $\mathrm{BS}(1,1)$ satisfies Statement~3. So, further, we can assume that $G$ is~non-solv\-a\-ble and~the~labeled graph~$\mathcal{L}(\Gamma)$ defining it~is~reduced. Then, by~Theorem~\ref{theorem03}, $\operatorname{Im}\Delta = \{1\}$ and~$|\lambda(\varepsilon e)| = 1$ for~all $e \in E$, $\varepsilon = \pm 1$. This means that $\Gamma$ has one vertex~$v$ and~$G$ is~an~extension of~the~vertex group~$G_{v}$ by~the~free group generated by~the~elements~$t_{e}$ ($e \in E$). Since such an~extension is~splittable, $G$ contains a~free subgroup~$F$ such that $G = G_{v}F$ and~\mbox{$G_{v} \cap F = 1$}. It~follows from~the~equality $\operatorname{Im}\Delta = \{1\}$ and~Proposition~\ref{proposition404} that $G_{v}$ lies in~the~center of~$G$. Therefore, $G = G_{v} \times F$, as~required.

$3 \Rightarrow 2$.\hspace{1ex}The~direct product of~two free groups is~residually free by~\cite[Lemma~1.1]{Gruenberg1957}.

$2 \Rightarrow 1$.\hspace{1ex}It is~well known that, for~an~arbitrary free group, the~intersection of~the~members of~its lower central series is~trivial~\cite{Magnus1935} and~the~factors of~this series are free abelian groups~\cite{HallM1950}. Therefore, every free group is~residually a~tor\-sion-free nilpotent group.
\end{proof}

\section*{Acknowledgements}

The~author would like to~thank F.~A.~Dudkin (Sobolev Institute of~Mathematics, Russia) for~the~introduction to~modern studies of~generalized Baumslag--Solitar groups.

\end{document}